\documentclass[11pt]{article}
\usepackage{enumerate}
\usepackage{hyperref}
\usepackage{amsmath}
\usepackage{amssymb,amsthm,authblk}
\usepackage{tikz}
\usepackage{booktabs}
\usetikzlibrary{shapes,backgrounds}
\usetikzlibrary{arrows,positioning} 
\hypersetup{
    pdftitle=   {Branch-depth: Generalizing tree-depth of graphs},
   pdfauthor=  {Matt DeVos, O-joung Kwon, and Sang-il Oum}
 }
\newcommand\abs[1]{\lvert #1\rvert}

\newcommand\decomposition{de\-com\-po\-si\-tion}
\newtheorem{THM}{Theorem}[section]
\newtheorem*{THMWQO}{Theorem~\ref{thm:main-wqo}}
\newtheorem*{THMTD}{Theorem~\ref{thm:main-td}}
\newtheorem*{THMSD}{Theorem~\ref{thm:main-sd}}
\newtheorem*{THMLC}{Theorem~\ref{thm:main-long-circuit}}
\newtheorem{LEM}[THM]{Lemma}
\newtheorem{COR}[THM]{Corollary}
\newtheorem{PROP}[THM]{Proposition}

\newcommand\F{\mathbb F}
\newcommand\M{\mathcal M}
\newcommand\N{\mathcal N}

\newcommand\I{\mathcal I}
\newcommand\cd{\operatorname{cd}}
\newcommand\dd{\operatorname{dd}}
\newcommand\cdd{\operatorname{cdd}}
\newcommand\dist{\operatorname{dist}}
\newcommand\td{\operatorname{td}}
\newcommand\bd{\operatorname{bd}}
\newcommand\rd{\operatorname{rd}}

\begin{document}
\title{Branch-depth: Generalizing tree-depth of graphs}
\author{Matt DeVos\thanks{Supported by an NSERC Discovery Grant (Canada).}}
\affil{Department of Mathematics, Simon Fraser University, Burnaby,~Canada}
\author[23]{O-joung Kwon\thanks{Supported by the National Research Foundation of Korea (NRF) grant funded by the Ministry of Education (No. NRF-2018R1D1A1B07050294).}\thanks{Supported by the Institute for Basic Science (IBS-R029-C1).}}
\affil{Department of Mathematics, Incheon National University, Incheon,~Korea}
\author[$\dagger$34]{Sang-il Oum}
\affil[3]{Discrete Mathematics Group,
  Institute for Basic Science (IBS), Daejeon,~Korea}
\affil[4]{Department of Mathematical Sciences, KAIST, Daejeon,~Korea}
\affil[ ]{\small \texttt{mdevos@sfu.ca}, \texttt{ojoungkwon@gmail.com}, \texttt{sangil@ibs.re.kr}}
\date{\today}
\maketitle
\begin{abstract}
We present a concept called the \emph{branch-depth} of a connectivity function, that generalizes the tree-depth of graphs.
Then we prove two theorems showing that this concept aligns closely with the notions of tree-depth and shrub-depth of graphs as follows.
For a graph $G = (V,E)$ and a subset $A $ of $E$ we let $\lambda_G (A)$ be the number of vertices incident with an edge in $A$ and an edge in $E \setminus A$.
For a subset $X$ of $V$, 
let $\rho_G(X)$ be the rank of the adjacency matrix between $X$ and $V \setminus X$ over the binary field.
We prove that a class of graphs has bounded tree-depth if and only if the corresponding class of functions $\lambda_G$ has bounded branch-depth
and similarly a class of graphs has bounded shrub-depth if and only if the corresponding class of functions $\rho_G$ has bounded branch-depth, which we call
the rank-depth of graphs.

Furthermore we investigate various potential generalizations of
tree-depth to matroids
and prove that matroids representable over a fixed finite field
having no large circuits 
are well-quasi-ordered by restriction.
\end{abstract}

\section{Introduction}\label{sec:intro}
The tree-width is one of the best-studied width parameters
for graphs, measuring how close a graph is to being a tree.
Similarly, the tree-depth is a `depth parameter' of graphs, 
measuring how close a graph is to being a star. We remark that graphs in this paper are allowed to have parallel edges and loops. A \emph{simplification} of a graph $G$ is a simple graph obtained from $G$ by deleting parallel edges and loops.

Definitions of both tree-width and tree-depth of graphs
involve vertices and edges
and so it is non-trivial to extend them to other discrete structures, such as matroids.
Actually Hlin\v{e}n\'y and Whittle~\cite{HW2003} were able to generalize the tree-width of graphs to matroids but it is not so obvious why it is related to the tree-width of graphs.

One may ask whether we can define an alternative depth parameter
of graphs that only uses some connectivity function.
If so, then it
can be easily extended to matroids.

For tree-width, this was done by Robertson and Seymour~\cite{RS1991}.
They defined the branch-width of graphs
and showed that a class of graphs has bounded branch-width
if and only if it has bounded tree-width.
Following Ding and Oporowski~\cite{DO1996},
we will say that two parameters $\alpha$, $\beta$ of some objects
are \emph{tied}
if there exists a function $f$ such that for every object $G$,
$\alpha(G)\le f(\beta(G))$ and $\beta(G)\le f(\alpha(G))$.
So branch-width and tree-width
are tied for graphs. 

The decomposition tree for branch-width of graphs, called the  \emph{branch-\decomposition}, is defined in terms of edges only.
When measuring the width of a branch-\decomposition,
it uses  a function $\lambda_G$ of a graph $G$ defined as follows: for a set $X$ of edges, $\lambda_G(X)$ is the number of vertices incident with both an edge in $X$ and an edge not in $X$. It turns out that $\lambda_G$ has all the nice properties such as symmetricity and submodularity and so such a function will be called a \emph{connectivity function}. The detailed definition will be reviewed in Section~\ref{sec:def}.
Robertson and Seymour~\cite{RS1991} defined branch-width not only for graphs
but also for any discrete structures admitting a connectivity function.

We will define the \emph{branch-depth} of a connectivity function
in Section~\ref{sec:def}.
This single definition will produce various depth parameters
for graphs and matroids,
by plugging in different connectivity functions.
We summarize them in Table~\ref{tab:depth}.

\begin{table}
  \centering
  \begin{tabular}{cccc}
    \toprule
    Object
    & 
    \parbox[center]{5em}{Connectivity\\function $f$}
      &Branch-width of $f$
    & Branch-depth of $f$\\
    \midrule
    Graph $G$
    & $\lambda_G$ & Branch-width of $G$ & Branch-depth of $G$
    \\
    Matroid $M$
    & $\lambda_M$  
    &  Branch-width of $M$ & Branch-depth  of $M$
    \\
    Graph $G$
    & $\rho_G$ & Rank-width of $G$  & Rank-depth of $G$\\
    \bottomrule
  \end{tabular}
  \caption{Instances of branch-width and branch-depth. Branch-width of graphs and matroids and rank-width of graphs are defined as branch-depth of corresponding connectivity functions. Similarly branch-depth of graphs and matroids and rank-depth of graphs are defined as branch-depth of corresponding connectivity functions.}
  \label{tab:depth}
\end{table}

Our first goal is to make sure that these new depth parameters
are tied to existing ones.
First we show that for graphs,
branch-depth and tree-depth are tied 
in Section~\ref{sec:treedepth}.
\begin{THMTD}
  Let $G$ be a graph,
  $k$ be its branch-depth,
  and  $t$ be its tree-depth. Then
  \[
  k-1\le t\le \max(2k^2-k+1,2).
  \]
\end{THMTD}
For a simple graph $G$, there is another interesting connectivity function
called the \emph{cut-rank function}, denoted by $\rho_G$. It is defined in terms of the rank
function of a submatrix of the adjacency matrix, which we will describe in detail in Section~\ref{sec:shrubdepth}.
Oum and Seymour~\cite{OS2004}
used the cut-rank function to define the rank-width of a simple graph.
Analogously we use the cut-rank function
to define the \emph{rank-depth} of a simple graph $G$
as the branch-depth of its cut-rank function $\rho_G$.
As a similar attempt to define a depth parameter comparable to rank-width,
Ganian, Hlin\v{e}n\'y, Ne\v{s}et\v{r}il, Obdr\v{z}\'alek, and
Ossona de Mendez~\cite{GHNOO2017}
defined the shrub-depth of a class of simple graphs.
In Section~\ref{sec:shrubdepth} we prove that 
having bounded rank-depth is equivalent to having bounded shrub-depth.
\begin{THMSD}
A class of simple graphs has bounded rank-depth if and only if it has bounded shrub-depth.
\end{THMSD}

For matroids, we define the branch-depth of a matroid $M$
as the branch-depth of its connectivity function $\lambda_M$.
We will investigate three additional depth parameters, called the contraction depth,
the deletion depth, and the contraction-deletion depth for matroids.
Ding, Oporowski, and Oxley~\cite{DOO1995} investigated the contraction-deletion depth under the name \emph{type}.
A Venn diagram in Figure~\ref{fig:matroid}  shows containment relation and properties on 
classes of matroids with respect to the boundedness of various parameters
of matroids,
which will be proved in Theorems~\ref{thm:various-depth} and \ref{thm:cd},
and Corollaries~\ref{cor:dd} and \ref{cor:cddd}.
We will also show that no two of these depth parameters are tied.

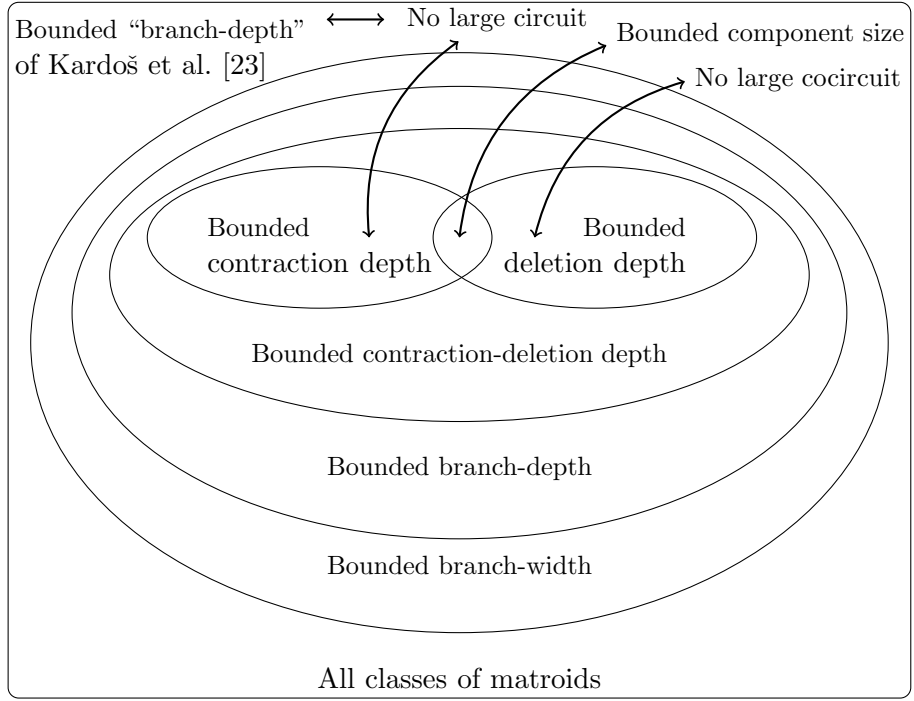
\begin{figure}[t]
  \centering
  \begin{tikzpicture}
\node[rounded corners, draw,text height = 9cm,minimum width=12cm] (main){All classes of matroids};
\node[ellipse, draw,  text height =5.2cm, yshift=.1cm,minimum width = 11.4cm,label={[anchor=south,above=6.5mm]270:\small Bounded branch-width}] at (main.center)  (semi) {};
\node[ellipse, draw,  yshift=.5cm, text height =4cm, minimum width = 10.3cm,label={[anchor=south,above=6.5mm]270:\small Bounded branch-depth}] at (main.center)  (semi) {};
\node[ellipse, draw, text height = 2.5cm, minimum width = 9.3cm,yshift=1cm,label={[anchor=south,above=6mm]270:\small Bounded contraction-deletion depth}] at (main.center) (active) {};
\node[ellipse, draw, align=right, anchor=east,minimum width=4.3cm,
text height=1cm,yshift=.5cm,
xshift=-0.7cm ] at (active.east) (non) {\small Bounded\\deletion depth};
\node[ellipse, draw,align=left, anchor=west, minimum width=4.3cm,
text height=1cm,yshift=.5cm,xshift=.5cm] at (active.west) (non2) {\small Bounded\\contraction depth};
\node at (4,4.2) {\small Bounded component size} edge [<->,thick, bend right=30] (0,1.5) ;
\node at (4.5,3.6) {\small No large cocircuit} edge [<->,thick, bend right=30] (1,1.5) ;
\node at (.5,4.4) {\small No large circuit} edge [<->,thick, bend right=30] (-1.2,1.5) ;
\node at (-1.9,4.4) [label={[below left,align=left]\small Bounded 
``branch-depth'' \\
of Kardo\v{s} et al.~\cite{KKLM2017}}] {} edge [<->,thick](-1, 4.4); 
  \end{tikzpicture}
  \caption{A Venn diagram of classes of matroids.}
  \label{fig:matroid}
\end{figure}

It is well known that a class of graphs has bounded tree-depth if and only if
it has no long path, see~\cite[4.4]{RS1985} or~\cite[Proposition 6.1]{NO2012}.
It is straightforward to prove the following analogue for matroids as follows.
\begin{THMLC}
  \begin{enumerate}[(i)]
  \item   A class of matroids has bounded contraction depth if and only if
  all circuits have bounded size.
  \item   A class of matroids has bounded deletion depth if and only if
  all cocircuits have bounded size.
  \end{enumerate}
\end{THMLC}

The contraction depth is also tied to the ``branch-depth''
introduced by Kardo\v{s}, Kr\'al', Liebenau, and Mach~\cite{KKLM2017}.
The authors would like to apologize for this unfortunate conflict of terms.
To distinguish with our branch-depth,
let us call theirs the \emph{KKLM-depth} of a matroid.
They prove that if $c$ is the size of a largest circuit of a matroid $M$,
then the KKLM-depth is at most $c^2$
and at least $\log_2 c$. By Theorem~\ref{thm:main-long-circuit},
KKLM-depth and contraction depth are tied.
Interestingly, they showed that unlike the contraction depth, KKLM-depth does not increase
when we take a minor. Moreover there is an interesting algorithmic result on integer programming parameterized by their branch-depth due to Chan, Cooper, Kouteck\'y, Kr{\'a}l', and Pek{\'a}rkov{\'a}~\cite{CCKKP2019}.

Robertson and Seymour~\cite{RS1985} stated that graphs of bounded tree-depth
are well-quasi-ordered by the subgraph relation
and Ding~\cite{Ding1992} proved that they are well-quasi-ordered by the induced subgraph relation. %
As a generalization, 
Ganian, Hlin{\v{e}}n{\'y}, Ne{\v{s}}et{\v{r}}il, Obdr\v{z}\'{a}lek, and Ossona de Mendez~\cite{GHNOO2017} proved that a class of graphs of bounded shrub-depth is well-quasi-ordered by the induced subgraph relation.
We prove an analogous theorem for matroids as follows.
A matroid $N$ is a \emph{restriction} of a matroid $M$
if $N=M\setminus X$ for some $X$.
A class $\mathcal M$ of matroids is \emph{well-quasi-ordered} by restriction
if every infinite sequence $M_1$, $M_2$, $\ldots$ of matroids in $\mathcal M$
has a pair $i<j$ such that $M_i$ is isomorphic to a restriction of $M_j$.

\begin{THMWQO}
  Let $\F$ be a finite field.
  Every class of $\F$-representable matroids of bounded contraction depth
  is well-quasi-ordered by restriction.
\end{THMWQO}

The paper is organized as follows.
Section~\ref{sec:def} introduces the branch-depth of a connectivity function.
Then we discuss various connectivity functions arising from graphs and matroids.
Sections~\ref{sec:treedepth} and \ref{sec:shrubdepth} are for graphs;
Section~\ref{sec:treedepth} introduces the branch-depth of a graph
by using the usual connectivity function of graphs defining branch-width.
Section~\ref{sec:shrubdepth} introduces the rank-depth of a simple graph
by using the cut-rank function.
In Section~\ref{sec:matroid} we introduce the branch-depth of a matroid
as the branch-depth of the matroid connectivity function.
We also investigate properties of the branch-depth, 
the contraction depth, the deletion depth, and the contraction-deletion depth
and obtain an inequality between the rank-depth and the tree-depth of simple graphs.
Section~\ref{sec:wqo} proves that matroids representable over a fixed finite field having
bounded contraction depth are well-quasi-ordered by restriction.

\section{Branch-Depth of a connectivity function}
\label{sec:def}
\subsection{Connectivity functions}\label{subsec:conn}
For a finite set $E$, we let $2^E$ denote the set of all subsets of $E$.  
A \emph{connectivity function} $\lambda$ on $E$
is a function $\lambda:2^E\to \mathbb{Z}$ satisfying the following three conditions.
\begin{enumerate}[(i)]
\item $\lambda(\emptyset)=0$
\item (\emph{symmetric}) $\lambda(X)=\lambda(E\setminus X)$ for all $X\subseteq E$.
\item (\emph{submodular}) $\lambda(X)+\lambda(Y)\ge \lambda(X\cup Y)+\lambda(X\cap Y)$ for all $X,Y\subseteq E$.
\end{enumerate}

\begin{LEM}\label{lem:disjointunion}
  Let $\lambda$ be a connectivity function on $E$ and 
  let $K$ be a subset of~$E$ such that $\lambda(K)=0$.
  Then the function $\lambda|_K$ on $2^K$ defined by $\lambda|_K(X)=\lambda(X)$
  is a connectivity function on $K$
  and 
  \[\lambda(X)=\lambda|_K(X\cap K)+\lambda|_{E\setminus K}(X\setminus K)\]
  for all $X\subseteq E$.
\end{LEM}
\begin{proof}
  Clearly $\lambda|_K$ is submodular and $\lambda|_K(\emptyset)=0$.
  By symmetry and submodularity, for all $X\subseteq K$, 
  \[\lambda(X)=\lambda(K)+\lambda(E\setminus X)\ge \lambda(K\setminus X)+\lambda(E)=\lambda(K\setminus X)\]
  and therefore $\lambda(X)\ge \lambda(K\setminus X)$. 
  By replacing $X$ with $K\setminus X$, we deduce that $\lambda(X)=\lambda(K\setminus X)$.
  Thus $\lambda|_K$ is a connectivity function on $E$.

  Let $Y$ be a subset of $E$. Then by submodularity,
  \[\lambda(Y\cap K)+\lambda(Y\setminus K)\ge \lambda(Y)+\lambda(\emptyset)=\lambda(Y).\]
  Again by submodularity, $\lambda(Y)=\lambda(Y)+\lambda(K)\ge \lambda(Y\cap K)+\lambda(Y\cup K)=\lambda(Y\cap K)+\lambda(E\setminus (Y\cup K))$.
  Since $\lambda(E\setminus K)=0$, $\lambda|_{E\setminus K}$ is a connectivity function
  and therefore  $\lambda(E\setminus (Y\cup K))=\lambda(Y\setminus K)$.
  Thus $ \lambda(Y)\ge \lambda(Y\cap K)+\lambda(Y\setminus K)$
  and so we deduce that 
  $\lambda(Y)=\lambda(Y\cap K)+\lambda(Y\setminus K)=\lambda|_{K}(Y\cap K)+\lambda|_{E\setminus K } (Y\setminus K)$.
\end{proof}

We say that a connectivity function $\lambda$ on $E$ is the
\emph{disjoint union} of two connectivity functions $\lambda_1$,
$\lambda_2$ on $E_1$, $E_2$ respectively, if $E=E_1\cup E_2$, $E_1\cap
E_2=\emptyset$, 
and $\lambda(X)=\lambda_1(X\cap E_1)+\lambda_2(X\cap E_2)$.

For a partition $\mathcal P$ of $E$ into subsets,
we define 
\[
  \lambda(\mathcal P):=\max_{\mathcal P'\subseteq \mathcal P} \left(
    \bigcup_{X\in P'} X
    \right).
\]
This function was introduced by Geelen, Gerards, and Whittle~\cite{GGW2006}. We will use it to measure the `connectivity' of a partition.
\subsection{Branch-depth}
A \emph{radius} of a tree is the minimum $r$ such that there is a node having distance at most $r$ from every node.

A \emph{decomposition} of a connectivity function $\lambda$ on $E$ is 
a pair $(T,\sigma)$ of a tree $T$ with at least one internal node
and a bijection $\sigma$ from $E$ to the set of leaves of $T$.
The \emph{radius} of a decomposition $(T,\sigma)$ 
is defined to be the radius of the tree $T$. 
For an internal node $v \in V(T)$, the components of the graph $T - v$ give rise to a partition $\mathcal P_v$ of $E$ by $\sigma$. 
The \emph{width} of $v$ is defined to be $\lambda(\mathcal P_v)$.
The \emph{width} of the decomposition $(T,\sigma)$ is the maximum width of an internal node of $T$.  We say that 
a decomposition $(T,\sigma)$ is a $(k,r)$-\emph{\decomposition} of $\lambda$ if the width is at most $k$ and the radius is at most~$r$.
The \emph{branch-depth} of $\lambda$ is the minimum $k$ such that there exists a $(k,k)$-\decomposition{} of $\lambda$.

If $\abs{E}<2$, then there exists no decomposition and we define $\lambda$ to have branch-depth $0$.
Note that every tree in a decomposition has radius at least~$1$ and therefore
the branch-depth of $\lambda$ is at least $1$ if $\abs{E}>1$.

\begin{LEM}\label{lem:connsub}
  Let $\lambda$ be a connectivity function on $E$ and 
  let $K$ be a subset of $E$ such that $\lambda(K)=0$.
  Let $k$, $r_1$, $r_2$ be integers such that 
   $\lambda|_{K}$ has a $(k,r_1)$-decomposition
  and  $\lambda|_{E\setminus K}$ has a $(k,r_2)$-decomposition.
  \begin{enumerate}[(i)]
  \item   If $r_1\neq r_2$, then 
  $\lambda$ has a $(k,\max(r_1,r_2))$-decomposition. 
\item   If $r_1=r_2$, then 
  $\lambda$ has a $(k,r_1+1)$-decomposition. 
  \end{enumerate}
\end{LEM}
\begin{proof}
  We may assume that $r_1\ge r_2$.
  Let $(T_1,\sigma_1)$ be a $(k,r_1)$-decomposition of $\lambda|_{K}$
  and let $v_1$ be an internal node of $T_1$ such that each node of $T_1$
  is within distance $r_1$ from $v_1$.
  Let $(T_2,\sigma_2)$ be a $(k,r_2)$-decomposition of $\lambda_{E\setminus K}$
  and let $v_2$ be an internal node of $T_2$ such that each node of $T_2$
  is within distance $r_2$ from~$v_2$.
  
  Let $T$ be the tree obtained from the disjoint union of $T_1$ and $T_2$ by
  adding an edge between $v_1$ and $v_2$.
  Then $T$ with $\sigma_1$ and  $\sigma_2$ forms a decomposition  $(T,\sigma)$
  of $\lambda$.
    
  Then $\lambda$ is the disjoint union of $\lambda_1$ and $\lambda_2$ and by Lemma~\ref{lem:disjointunion}, $(T,\sigma)$ has width at most $k$.
  There are two cases for the radius of $T$.
  If $r_2<r_1$, then the radius of $T$ is at most $r_1$
  and therefore $(T,\sigma)$ is a $(k,r_1)$-decomposition.
  If $r_2=r_1$, then the radius of $T$ is at most $r_1+1$
  and so $(T,\sigma)$ is a $(k,r_1+1)$-decomposition.
\end{proof}
\begin{LEM}\label{lem:conn}
  Let $\lambda$ be a connectivity function on $E$.
  Let $E_1, E_2,\ldots,E_m$ be  a partition of $E$ into non-empty sets
  such that $\lambda(E_i)=0$ for all $1\le i\le m$.
  Let $\lambda_i:=\lambda_{|_{E_i}}$ and 
  $k_i$ be the branch-depth of $\lambda_i$.
  Let $k=\max(k_1,k_2,\ldots,k_m)$.

  Then the branch-depth of $\lambda$ is $k$ or $k+1$.
  In particular, if the branch-depth of $\lambda$ is $k+1$,
  then there exist $i<j$ such that $k_i=k_j=k$
  and $\lambda$ has a $(k,k+1)$-decomposition.
\end{LEM}
\begin{proof}
  We proceed by induction on $m$.   We may assume that $m\ge 2$.
  Clearly
  the branch-depth of $\lambda$ is at least the branch-depth of $\lambda_i$,
  simply by taking  a subtree
  and therefore the branch-depth of $\lambda$ is at least $k$.

  If $k_i=0$ for all $i$, then $\abs{E_i}=1$ for all $i$.
  Then  the branch-depth of $\lambda$ is~$1$
  because there is a $(0,1)$-decomposition $(T,\sigma)$
  where $T$ is $K_{1,m}$.
  
  Thus we may assume that $k=k_1>0$.
  Let $(T_1,\sigma_1)$ be a $(k,k)$-\decomposition{} of $\lambda|_{E_1}$.
  Let $r_1$ be an internal node of $T_1$ such that each node of $T_1$ is within distance $k$ from $r_1$. 

  If $\abs{E_2}=1$, then
  we attach to $r_1$ a leaf node corresponding to the element of $E_2$, 
  producing a $(k,k)$-decomposition of $\lambda_{|_{E_1\cup E_2}}$.
  By applying the induction hypothesis to $E_1\cup E_2, E_3,\ldots,E_m$,
  we deduce the conclusion.
  Thus we may assume that $\abs{E_2},\abs{E_3},\ldots,\abs{E_m}\ge 2$
  and therefore $k_2,k_3,\ldots,k_m\ge 1$.
  Now it is trivial to deduce the conclusion by using Lemma~\ref{lem:connsub} repeatedly.
\end{proof}
\section{Branch-depth and tree-depth of graphs}\label{sec:treedepth}
Recall that for a graph $G=(V,E)$ and a set $X\subseteq E$,  $\lambda_G(X)$ is the number of vertices incident with both an edge in $X$ and an edge in $E\setminus X$.
Let us define the \emph{branch-depth} of a graph $G$ to be the branch-depth of $\lambda_G$. Let us write $\bd(G)$ to denote the branch-depth of $G$.
In this section we aim to prove that tree-depth and branch-depth are tied for graphs.

First let us review the definition of tree-depth \cite[Chapter 6]{NO2012}.
A \emph{rooted forest} is a forest in which every component has  a specified node called a \emph{root} and each edge is oriented away from a root. The \emph{closure} of a rooted forest $T$ is the undirected simple graph on $V(T)$ in which two vertices are adjacent if and only if there is a directed path from one to the other in $T$.
The \emph{height} of a rooted forest is the number of vertices
in  a longest directed path.
The \emph{tree-depth} of a simple graph  $G$, denoted by $\td(G)$, is
the minimum height of a rooted forest whose closure contains $G$ 
as a subgraph. The \emph{tree-depth} of a graph is defined as the tree-depth of its simplification.

Unlike tree-depth, the branch-depth of a graph may be different from the branch-depth
of its simplification. For instance,
a one-vertex graph has branch-depth $1$ if it has at least two loops,
and branch-depth $0$ otherwise.

Now let us see why a class of graphs of bounded tree-depth has bounded branch-depth. The following lemma proves one direction.

\begin{LEM}\label{lem:td2bd}
  The branch-depth of a connected graph is less than or equal to its tree-depth.
\end{LEM}
\begin{proof}
  Let $G$ be a graph. If $G$ has at most one edge, then $G$ has branch-depth $0$ and therefore we may assume that $G$ has at least two edges.
  
  Let $k$ be the tree-depth of $G$.
  Let $F$ be a rooted forest of height $k$ whose closure contains
  the simplification of $G$ as a subgraph.
  We may assume that $V(F)=V(G)$.
  Thus $F$ is connected.

  Let $T$ be the tree obtained from $F$ 
  by attaching 
  one leaf  to a node~$v$ 
  for every edge $e=uv$ of $G$ when $v$ is under $u$ in $F$.
  (If $e$ is a loop incident with $v$, then we attach a leaf to $v$.)
  Let $\sigma$  be the bijection from $E(G)$ to the leaves of $T$ given by the construction of $T$.
  
  We claim that $(T,\sigma)$ is a $(k,k)$-decomposition.
  Let $v$ be an internal node of $T$, located at the distance $i$ from the root.
  Then $i\le k-1$ because the height of $F$ is $k$.
  Let $\mathcal P$ be the partition of $E(G)$ given by the node $v$.
  Then if a vertex of $G$ meets more than one part of $\mathcal P$, then 
  it is in the path from the root of $T$ to $v$ and therefore 
  the width of $v$ is at most $i+1\le k$.
  Observe that the radius of $T$ is at most $k$.
  This proves our claim.
\end{proof}

Now let us prove the backward direction. We will need the following lemma. 

\begin{LEM}\label{lem:treedepthlem}
  Let $k$ be an integer.
Let $G = (V,E)$ be a graph and let $\mathcal{P}$ be a partition of $E$.
If $\lambda_G(\mathcal P)\le k$, then 
$G$ has at most $\max(2k-1,0)$ vertices incident with  edges from at least two parts of $\mathcal{P}$.
\end{LEM}

\begin{proof}
  Choose a subset $\mathcal{P'}$ of $\mathcal{P}$ 
  by selecting each part of $\mathcal{P}$ independently at random
  with probability $1/2$.
  Let $A = \bigcup_{F\in \mathcal{P'}} F$ and let $Y$ denote the set of vertices 
  incident with both an edge in $A$ and an edge in $E\setminus A$.

  Let $X$ be the set of vertices meeting edges in at least two parts of $\mathcal P$.
  We may assume that $\abs{X}>0$.
  Every vertex in $X$
  will appear in $Y$ with probability at least $1/2$ since this vertex is incident with edges from at least two parts of $\mathcal{P}$.  
  It then follows from linearity of expectation that 
  $\mathbb{E}[ \abs{Y}] \ge \frac{1}{2} \abs{X}$.  In particular,
  because $P(\abs{Y}<\frac{1}{2}\abs{X})\ge P(\mathcal P'=\emptyset)>0$, 
  there exists a subset $\mathcal{P'}$ of $\mathcal P$ 
  for which $\abs{Y}> \frac{1}{2} \abs{X}$.
  By our assumption, we must have $\abs{Y} \le k$ and this gives the desired bound.
\end{proof}

\begin{LEM}\label{lem:bd2td}
  Let $r$, $w$ be positive integers.
  If a graph has a $(w,r)$-\decomposition, then 
  it has tree-depth at most $(2w-1)r+1$.
\end{LEM}
\begin{proof}
  We proceed by the induction on $r$.
  Let $G$ be a graph having a $(w,r)$-\decomposition{} $(T,\sigma)$.
  Let $v$ be an internal node of $T$ such that 
  each node of $T$ is within distance $r$ from~$v$.

  If $r=1$, then by Lemma~\ref{lem:treedepthlem},
  the number of vertices of degree at least two in $G$ is at most $2w-1$
  and thus the tree-depth is at most $2w$.
  Now we assume that $r>1$.

  Let $\mathcal P$ be a partition of $E(G)$ induced by $T-v$.
  Let $X$ be the set of vertices of $G$ 
  meeting at least two parts of $\mathcal P$.
  By Lemma~\ref{lem:treedepthlem}, $\abs{X}\le 2w-1$.

  It is enough to prove that $G\setminus X$ has tree-depth at most $(2w-1)(r-1)+1$.
  Let $C$ be a component of $G\setminus X$.
  If $C$ has at most two edges,  then the tree-depth of $C$ is at most $2$.
  As $r>1$, $(2w-1)(r-1)+1\ge 2$.
  If $C$ has at least three edges, then let $T'$ be the minimal subtree of $T$ containing the leaves of $T$ associated with edges in $C$. Since $T'$ does not contain $v$,  the radius of $T'$ is at most $r-1$. Furthermore, as $C$ has at least three edges, $T'$ has at least one internal node. Thus, $T'$ induces a $(w,r-1)$-decomposition of $C$. By induction hypothesis, $C$ has tree-depth at most $(2w-1)(r-1)+1$.
\end{proof}
Now we are ready to prove Theorem~\ref{thm:main-td}.
\begin{THM}\label{thm:main-td}
  Let $G$ be a graph,
  $k$ be its branch-depth,
  and  $t$ be its tree-depth. Then
  \[
  k-1\le t\le \max(2k^2-k+1,2).
  \]
\end{THM}
\begin{proof}[Proof of Theorem~\ref{thm:main-td}]
  By Lemma~\ref{lem:conn},
  there is a component $C$ of $G$ whose branch-depth is $k$ or $k-1$.
  By Lemma~\ref{lem:td2bd},
  that component has tree-depth at least $k-1$
  and so $\td(G)\ge k-1$.

  If $k=0$, then it has at most $1$ edge
  and so its tree-depth is at most $2$.
  If $k>0$, then by Lemma~\ref{lem:bd2td},
  we have $\td(G)\le (2k-1)k+1=2k^2-k+1$.
\end{proof}

We say that a graph $H$ is a \emph{minor} of a graph $G$ if
$H$ can be obtained from $G$ by contracting edges and deleting edges and vertices.
As many graph parameters do not increase under taking minors, we may wonder whether the branch-depth behaves similarly. It is easy to prove the following proposition. 
\begin{PROP}
  If $H$ is a minor of a graph $G$, then $\bd(H)\le \bd(G)$.
\end{PROP}

Ding~\cite{Ding1992} proved that graphs of bounded tree-depth are well-quasi-ordered under the induced subgraph relation. This implies that for each $k$, there is a finite list of graphs such that a graph $G$ has branch-depth at most $k$ if and only if no graph in the list is isomorphic to an induced subgraph of $G$.
Since no graph of large tree-width has small branch-depth, 
one can decide in linear time whether the input graph has an induced subgraph isomorphic to a fixed graph, for instance by using Courcelle's theorem~\cite{Courcelle1990} with the algorithm to find a tree-decomposition by Bodlaender~\cite{Bodlaender1996}.
The authors are not aware of such an algorithm that in addition finds a decomposition for branch-depth of graphs.

\section{Rank-depth and shrub-depth of simple graphs}\label{sec:shrubdepth}
The \emph{cut-rank function} of a simple graph  $G = (V,E)$ is defined as a function $\rho_G$ on the subsets of $V$ such that 
$\rho_G(X)$ is the rank of an $X\times (V\setminus X)$ $0$-$1$ matrix 
\[A_X=(a_{ij})_{i\in X, j\in V\setminus X}\]
over the binary field where $a_{ij}=1$ if and only if $i$ and $j$ are adjacent and $0$ otherwise.
The cut-rank function is an instance of a connectivity function on the vertex set of a simple graph,
see a paper of Oum and Seymour~\cite{OS2004}.
We define the \emph{rank-depth} of a simple graph $G$, denoted by $\rd(G)$,
to be the branch-depth of~$\rho_G$.

As an example, we will prove that
\[ \rd(P_n)= \Theta(\log n/\log\log n),\]
where $P_n$ is the path graph on $n$ vertices.
Together with Theorem~\ref{thm:main-sd} to be proved later, this will give an alternative proof of the fact that 
a class of graphs containing arbitrary long induced paths has unbounded shrub-depth shown by 
Ganian, Hlin{\v{e}}n{\'y}, Ne{\v{s}}et{\v{r}}il, Obdr\v{z}\'{a}lek, and Ossona de Mendez~\cite{GHNOO2017}.
We write $\log$ to denote the natural logarithm $\log_{e}$ if the base is omitted.
\begin{PROP}\label{prop:path}
  For $n\ge 2$,
  \[ \rd(P_n)>\frac{\log n }{\log (1+4\log  n)}.\]
\end{PROP}
\begin{proof}
  Let $v_1,v_2,\ldots, v_n$ be the vertices of $P_n$ in the order.
  Let $\lambda_n$ be the cut-rank function of the path $P_n$.

  We claim that for a positive integer $k$,
  if $\lambda_n$ has a $(k,r)$-\decomposition{} $(T,\sigma)$,
  then \[r\ge \lceil \log_{4k+1} n\rceil.\]
  We proceed by induction on $n$.
  If $2\le n\le 4k+1$, then trivially the radius of $T$ is at least $1$.
  So we may assume that $n> 4k+1$.

  We may assume that the radius of $T$ is $r$.
  Let $v$ be an internal node of $T$ such that every node of $T$ is within distance $r$ from $v$.
  Let $e_1,e_2,\ldots,e_\ell$ be the edges incident with $v$ in $T$.
  We color a vertex $w$ of $P_n$ by $i\in\{1,2,\ldots,\ell\}$
  if the unique path from $v$ to $\sigma(w)$ on $T$ contains $e_i$.

  We say an edge of $P_n$ is \emph{colorful} if its ends have distinct colors.
  Let~$m$ be the number  of colorful edges of $P_n$.
  Then there exists a subset  $X$ of  $\{1,2,\ldots,\ell\}$
  such that at least $m/2$ colorful edges have exactly one end whose color
  is in $X$.
  We may assume that there are at least $m/4$ edges $v_iv_{i+1}$ 
  such that the color of $v_i$ is in $X$
  and the color of $v_{i+1}$ is not in $X$,
  because otherwise we may relabel vertices in the reverse order.

  Let $1\le i_1<i_2<\cdots<i_{\lceil m/4\rceil}< n$ be a  sequence of integers
  such that for each $j\in\{1,2,\ldots,\lceil m/4\rceil\}$, 
  the color of $v_{i_j}$ is in $X$
  and the color of $v_{i_j+1}$ is not in $X$.
  Then the submatrix of the adjacency matrix of $G$ consisting of rows from $v_{i_1}$, $v_{i_2}$, $\ldots$, $v_{i_{\lceil m/4\rceil}}$
  and columns from $v_{i_1+1}$, $v_{i_2+1}$, $\ldots$, $v_{i_{\lceil m/4\rceil}+1}$
  has rank exactly $\lceil m/4\rceil$
  because it is a triangular matrix with non-zero diagonal entries.
  This implies that the width of $(T,\sigma)$ is at least~$\lceil m/4\rceil$.
  So, $k\ge m/4$
  and therefore the number of colorful edges of $P_n$
  is at most $4k$.

  So there is a subpath $P_{n'}$ of $P_n$ having no colorful edges
  where $n'\ge \lceil\frac{n}{4k+1}\rceil>1$.
  Then $P_{n'}$ has a $(k,r-1)$-decomposition induced from $(T,\sigma)$.
  By the induction hypothesis,
  $r-1\ge \lceil \log_{4k+1}  n'\rceil\ge \log_{4k+1} n  - 1$.
  Thus $r\ge \log_{4k+1}n$. This proves the claim.

  From this claim we deduce that for a positive integer $k$, if \[k\log (4k+1)<\log n,\] then the rank-depth of $P_n$ is larger than $k$.

  Now suppose that $k= \lfloor \frac{\log n }{\log (1+4\log n)}\rfloor$.
  As $n\ge 2$, we have $2\log n\ge 1$ and therefore $\log (1+4\log n)> 1$. Thus $k< \log n$. 
  So we deduce that 
  \[
    k\log(4k+1)
    < \frac{\log n }{\log (1+4\log n)}
    \log (4\log n +1 )
    =\log n.
  \]
  Thus the rank-depth of $P_n$ is larger than $ {\log n }/{\log (1+4\log n)}$.
\end{proof}
\begin{PROP} For $n\ge 2$, 
\[ \rd(P_n)\le \left\lceil (1+o(1))\frac{\log n}{\log\log n}\right\rceil.\]
\end{PROP}
\begin{proof}
  For an integer $w$, consider any partition $\mathcal P$ of $V(P_n)$ into
  at most $w+1$ subpaths.
  We claim that for all $\mathcal P'\subseteq\mathcal P$,
  $\rho_G\left(\bigcup_{X\in\mathcal P'} X\right)\le w$.
  Let $U$ be the set of vertices $v$
  having a neighbor not in $X$ for $X\in \mathcal P$ with $v\in X$.
  Notice that each subpath $Q$ has at most two vertices having neighbors outside $Q$.
  In particular, if $Q$ contains the first or the last vertex of $P_n$, then it has
  at most one vertex having neighbors outside $Q$.
  This means that $\abs{U}\le 2(w-1)+2=2w$.
  Then $\rho_G\left(\bigcup_{X\in\mathcal P'} X\right) =
  \rho_{G[U]}\left(\bigcup_{X\in\mathcal P'}( X\cap U)\right) \le \abs{U}/2\le w$.

  We aim to construct a decomposition $(T,\sigma)$ of $P_n$
  such that each node has degree at most $w+1$
  and the radius is as small as possible
  so that at each internal node, the path is partitioned into at most $w+1$ subpaths.
  This can be done as follows:
  For the root node, we split the path into $w+1$ subpaths.
  For the non-root internal node, we split the path
  into at most~$w-1$ subpaths. Outside is split into at most two subpaths, one left and one right and therefore the path is partitioned into at most~$w+1$ subpaths at this internal node, as desired.
  It is clear that if
  \begin{equation}\label{eq:path}
    (w+1)(w-1)^{w-1}\ge n, 
  \end{equation}
  then the rank-depth of $P_n$ is at most $w$.

  Let $\ell=(1+\varepsilon)\frac{\log n}{\log\log n}$ for $\varepsilon>0$.
  We claim that if $w=\lceil \ell\rceil$ and $n$ is sufficiently large, then \eqref{eq:path} holds. Then,
  \[
    (w+1)(w-1)^{w-1}
    \ge \ell (\ell-1)^{\ell-1}
    = \ell^\ell \frac{1}{      (1+\frac{1}{\ell-1})^{\ell-1}}.
  \]
  We may assume that $n$ is sufficiently large so that
  $(1+\frac{1}{\ell-1})^{\ell-1}<e^2$. Then it is enough to show that
   $ \ell^\ell > e^2 n$ or equivalently $\ell\log \ell>\log n +2$
   for sufficiently large $n$. Note that
   \[
     \ell \log \ell
     = (1+\varepsilon) \log n  \left( 1 -
       \frac{\log\log\log n -\log (1+\varepsilon)}{\log\log n}\right).\]
   For sufficiently large $n$, 
   $ \frac{\log\log\log n -\log (1+\varepsilon)}{\log\log n} < \frac{\varepsilon}{2(1+\varepsilon)}$
   and therefore
   \[
     \ell\log\ell
     >   (1+\varepsilon) \left( 1 - \frac{\varepsilon}{2(1+\varepsilon)}\right)\log n
     =\left(1+\frac{\varepsilon}{2}\right)\log n.\]
   Thus, if $n$ is large enough, then $\frac{\varepsilon}{2}\log n>2$
   and therefore $\ell \log \ell > \log n +2$.
   This completes the proof of the claim, thus showing that $\rd(P_n)\le \lceil \ell\rceil$
   for sufficiently large $n$.
\end{proof}

There is an operation called the local complementation, that preserves the cut-rank
function of a simple graph.
For a simple graph $G$ and a vertex $v$, the \emph{local complementation} at $v$
is an operation to obtain a simple graph denoted by $G*v$ 
on the vertex set $V(G)$ from $G$
by removing edges between all adjacent pairs of neighbors of $v$
and adding edges between all non-adjacent pairs of 
neighbors of $v$.
We say that two simple graphs are \emph{locally equivalent} if one is obtained from the other
by a sequence of local complementations. 
It turns out that $G$ and $G*v$ share the identical cut-rank function, see Oum~\cite{Oum2004}. Thus if two simple graphs are locally equivalent, then
they have the same rank-depth.

In addition, it is easy to see that deleting vertices would never increase the rank-depth.
We say that $H$ is a \emph{vertex-minor} of $G$ if
$H$ is an induced subgraph of a simple graph locally equivalent to $G$.
Thus we deduce the following lemma.
\begin{LEM}\label{lem:vertexminor}
  If $H$ is a vertex-minor of $G$, then
  $\rd(H)\le \rd(G)$.
\end{LEM}

We aim to prove that a class of simple graphs has bounded rank-depth
if and only if
it has bounded shrub-depth, a concept
introduced by Ganian, Hlin\v{e}n\'y, Ne\v{s}et\v{r}il, Obdr\v{z}\'alek,
and Ossona de Mendez~\cite{GHNOO2017}.
We will review necessary definitions.

A $(k,d)$-\emph{shrubbery} for a graph $G=(V,E)$ consists of a rooted tree $T$
for which $V$ is the set of all leaves, together with a function $f : V \rightarrow \{1,\ldots,k\}$, with the property that
\begin{itemize}
\item all leaves
  of $T$ are distance exactly $d$ from the root, and 
\item adjacency in $G$ is completely determined by $f$ and the distance function in $T$.
  (In other words, whenever two pairs of vertices $(x_1,y_1), (x_2,y_2) \in V^2$ satisfy 
  \[\dist_T(x_1,y_1) = \dist_T(x_2,y_2),~f(x_1) = f(x_2)\text{, and }f(y_1) = f(y_2),\]
  we have $x_1 y_1 \in E$ if and only if $x_2 y_2 \in E$.)
\end{itemize}
We say that a class $\mathcal G$ of graphs has \emph{shrub-depth} $d$
if $d$ is the minimum integer such that
there exists an integer $k$ for which all graphs in $\mathcal G$ admit
a $(k,d)$-shrubbery.
Note that unlike many other parameters,  shrub-depth is not defined for a single graph but for a class of graphs.

From now on, we present several lemmas to prove Theorem~\ref{thm:main-sd}. 
Let us start with an easier direction, that is to 
show that
a class of simple graphs of bounded shrub-depth
has bounded rank-depth.
We say that
for a set $A$ of vertices of a simple graph $G$, 
a set $X\subseteq V(G)\setminus A$ of vertices is a set of \emph{clones relative to $A$}
if no vertex of $A$ has both a neighbor and a non-neighbor in~$X$.

\begin{LEM}\label{lem:shrubbery2decomp}
  Let $G$ be a simple graph with at least two vertices.
  If $G$ has a $(k,d)$-shrubbery, then its cut-rank function has a $(k,d)$-decomposition.
\end{LEM}
\begin{proof}
  Let $G = (V,E)$ and let $T$ together with the function $f : V
  \rightarrow \{1,\ldots k\}$ be a $(k,d)$-shrubbery for $G$.  We claim
  that $T$ is a $(k,d)$-\decomposition{} of $\lambda$.  By definition,
  the tree $T$ has height $d+1$.
  Now consider an internal node $v \in V(T)$ and let $\mathcal{P}$ be
  the partition of $V$ which is given by the components of $T - v$.

  Let $\mathcal{P'} \subseteq \mathcal{P}$ and
  define $A = \bigcup_{X  \in \mathcal{P'}} X$ and $B = V \setminus A$.
  We aim to show that $\rho_G(A)\le k$. By symmetry, we may assume
  that the components of~$T-v$ having $\mathcal P'$  do not contain the root of $T$ by swapping $\mathcal {P'}$ with
  $\mathcal P\setminus \mathcal P'$ if necessary.
  Then every vertex in $A$ has the same distance from~$v$ in~$T$. 

  For $1 \le i \le k$, we define 
  \[ A_{i} = \{ x \in A :  f(x) = j \}. \]
  Then every member of $A$ is in exactly one set of the form
  $A_{i}$.  It follows from the definition of shrubbery that
  $A_{i}$ is a set of clones relative to $B$.  Therefore, the
  cut-rank of $A$ is at most $k$.  We conclude that $T$ has width at
  most $k$ as a decomposition and its radius is at most $d$.
\end{proof}

Now it remains to show the converse that a class of simple graphs of bounded rank-depth
has bounded shrub-depth.
Roughly speaking, we will first describe edges joining parts of some partition $\mathcal P$ when $\rho_G(\mathcal P)$ is small. Then we will use some universal graph of bounded size to encode edges between distinct parts. That will allow us to assign colors to each vertex of a graph when building a shrubbery.
First we start with a lemma, whose proof uses the linearity of expectation.

\begin{LEM}\label{lem:root}
Let $G = (V,E)$ be a complete graph with $V = \{1,\ldots,n\}$ and let $f : E \rightarrow \{1,\ldots,n\}$ have the property that $f(e) \neq i$ whenever $e \in E$ and $i \in V$ are incident.  Then there exists a subset $S \subseteq V$ of size more than $\frac{1}{2}\sqrt{n}$ with the property that every edge $e$ with both ends in $S$ satisfies $f(e) \not\in S$.
\end{LEM}
\begin{proof}
  Choose a random subset $R \subseteq V$ by selecting each vertex independently with probability $1/\sqrt{n}$.  Consider the following subset of edges
  \[ B = \{ ij \in E : i, j, f(ij) \in R \}. \]
  The probability that a given edge is in $B$ is precisely $n^{-3/2}$ since it requires three vertices to be selected in $R$.  Now linearity of expectation gives us
  \[ \mathbb{E}[ \abs{R} - \abs{B} ] = \mathbb{E}[\abs{R}] - \mathbb{E}[\abs{B}] = n /\sqrt{n} - {n \choose 2} n^{-3/2} > \frac{1}{2}\sqrt{n}. \]
  So there exists a particular set $R$ for which $\abs{R} - \abs{B} > \frac{1}{2} \sqrt{n}$.  Now define the set 
  \[ S = R\setminus f(B). \]
  It follows immediately from this construction that every edge $e$ with both ends in $S$ satisfies $f(e) \not\in S$.  Furthermore, by construction $\abs{S} \ge \abs{R} - \abs{B} > \frac{1}{2} \sqrt{n}$, as desired.
\end{proof}

Let $G = (V,E)$ be a simple graph with cut-rank function $\rho_G : 2^V \rightarrow \mathbb{Z}$.  Let $\mathcal{P}$ be a partition of $V$, and define the \emph{cut-rank} of $\mathcal{P}$ to be $\rho_G(\mathcal P)$, that is the maximum over all $\mathcal{P'} \subseteq \mathcal{P}$ of $\rho_G( \bigcup_{X \in \mathcal{P'}} X )$.  
Let $A \in \{0,1\}^{V \times V}$ be the standard adjacency matrix for $G$ over 
the binary field 
and modify $A$ to form a new matrix $A'$ as follows.  For every $u,v
\in V$ if $u,v$ are in the same part of~$\mathcal{P}$ then replace
the $u,v$ entry of $A$ by $\star$.  We call the resulting matrix $A'$
the \emph{adjacency matrix for} $G$ \emph{relative to} $\mathcal{P}$.
A \emph{realization} of a vector consisting of $0$, $1$, and $\star$
is a vector obtained by replacing $\star$ with $0$ or $1$ arbitrary.

Our next lemma shows that this matrix has a simple structure when $\mathcal{P}$ has small cut-rank.  
For an $X\times Y$ matrix $A$ and $X'\subseteq X$, $Y'\subseteq Y$,
we write $A[X',Y']$ to denote the $X'\times Y'$ submatrix of $A$.

\begin{LEM}\label{lem:realization}
Let $G = (V,E)$ be a simple graph, let $\mathcal{P}$ be a partition of $V$ with cut-rank at most $k$, and let $A$ be the adjacency matrix of $G$ relative to~$\mathcal{P}$.  Then there exists a set $U \subseteq \{0,1\}^V$ of vectors  with $|U| \le 2^{2k}(2^{2k+2}-1)$ which contains a realization of every column of $A$.
\end{LEM}
\begin{proof}
Define two vertices $u,v$ to be \emph{similar} if the associated column vectors of $A$ can be modified to the same $\{0,1\}^V$ vector by replacing each $\star$ entry by either $0$ or $1$.  We say that $u,v$ are \emph{dissimilar} if they are not similar, and we now choose a maximal set $Z$ of pairwise dissimilar vertices with the property that no two vertices in $Z$ are in the same part of $\mathcal{P}$.  
For a set $X\subseteq V(G)$, we write $\bar{X}$ to denote $V(G)\setminus X$.

Suppose (for a contradiction) that $|Z| \ge 2^{2k+2}$ and let $Z = \{z_1, \ldots, z_m\}$.  For every $1 \le i \le m$ let $Z_i$ be the part of $\mathcal{P}$ which contains $z_i$.  Let $K_m$ be a complete graph with vertex set $\{1,\ldots,m\}$ and define a function $f : E(K_m) \rightarrow \{1,\ldots,m\}$ by the following rule.  If $1 \le i,j \le m$ and $i \neq j$ then the vertices $z_i$ and $z_j$ are dissimilar so there exists a vertex $x \not\in Z_i \cup Z_j$ so that the column vectors of $A$ associated with $z_i$ and $z_j$ differ in coordinate $x$.
If we can choose $x \in Z_t$ for some $1 \le t \le m$, then we define $f(ij) = t$.  Otherwise we assign $f(ij)$ to an arbitrary element of $\{1,\ldots,m\} \setminus \{i,j\}$.
By Lemma~\ref{lem:root}, we may choose a subset $S \subseteq \{1,\ldots m\}$ so that $\abs{S}> \frac{1}{2} \sqrt{2^{2k+2}}=2^k$ and $f(ij) \not\in S$ for every $i,j \in S$.  Now define $Y = \bigcup_{i \in S} Z_i$ and consider the matrix $A[{\bar{Y}}, Y]$.  It follows from our construction that for every distinct $i,j \in S$, the columns of this matrix associated with $z_i$ and $z_j$ are distinct; that is, there are $2^k+1$ pairwise distinct columns.  Then this matrix has rank at least $k+1$. But this contradicts our assumption on cut-rank.

Therefore we must have $\abs{Z} \le 2^{2k+2}-1$.  Let $X \in \mathcal{P}$ and note that by our cut-rank assumption, the matrix $A[{\bar{X}},X]$ has rank at most $k$.
It follows that $A[{\bar{X}},X]$ has at most $2^k$ distinct column vectors, and thus there are at most $2^k$ pairwise dissimilar vertices in $X$.  Therefore, if we choose a maximal set $Z'$ of pairwise dissimilar vertices with $Z \subseteq Z'$, then we will have $\abs{Z'} \le  2^k(2^{2k+2}-1)$.  

Now we form a set $U \subseteq \{0,1\}^V$ of vectors starting from the empty set by the following procedure.  For every $z \in Z'$ let $X$ be the part of $\mathcal{P}$ containing~$z$ and note that the matrix $A[\bar{X},X]$ has at most $2^k$ distinct columns.  Add to the set $U$ all vectors which may be obtained from the column vector of~$A$ associated with~$z$ by replacing the $\star$ entries with one of the columns from $A[\bar{X},X]$.  It follows from this construction that $\abs{U} \le 2^{2k}(2^{2k+2}-1)$ and every vector in $A$ may be turned into a vector in $U$ by replacing each $\star$ entries by either $0$ or $1$.  
\end{proof}

\begin{LEM}\label{lem:nodeuniversal}
If $G = (V,E)$ is a simple graph and $\mathcal{P}$ is a partition of $V$ with cut-rank at most $k$, then there exist a simple graph $H$ with \[\abs{V(H)} \le 2^{2k+1}(2^{2k+2}-1)\] and a function $h : V \rightarrow V(H)$ with the property that whenever $x,y \in V$ are in distinct parts of $\mathcal{P}$, we have $xy \in E$ if and only if $h(x) h(y) \in E(H)$.  
\end{LEM}
\begin{proof}
Let $A$ denote the adjacency matrix of $G$ relative to the partition $\mathcal{P}$ and apply Lemma~\ref{lem:realization} to choose a set $U = \{u_1, \ldots, u_m\} \subseteq \{0,1\}^V$ of vectors with $m \le 2^{2k}(2^{2k+2}-1)$.  Now choose a function $f_0 : V \rightarrow \{1,\ldots,m\}$ with the property that $f_0(v) = i$ implies that the column of $A$ associated with $v$ can be turned into the vector $u_i$ by replacing each $\star$ by either $0$ or $1$.  Next we modify $f_0$ to a new function $f : V \rightarrow \{1, \ldots, 2m \}$ by the following procedure.  If for $1 \le i \le m$ the set $f_0^{-1}(\{i\})$ has nonempty intersection with exactly two parts of $\mathcal{P}$, then we choose one such part, say $X$, and we define $f(x) = m+i$ for every $x \in X \cap f_0^{-1}(\{i\})$.  This newly constructed function $f$ still retains the property that $f(v) = f(v')$ only when the columns of $A$ associated with $v$ and $v'$ can be turned into the same $\{0,1\}^V$ vector by replacing the $\star$ entries.  However, it also has the property that $f^{-1}(\{i\})$ either intersects just one part of $\mathcal{P}$ or at least three parts of $\mathcal{P}$.  

We claim that whenever $x_1,x_2,y_1,y_2 \in V$ satisfy $f(x_1) = f(x_2)$ and $f(y_1) = f(y_2)$, and in addition $x_i,y_i$ are in distinct parts of $\mathcal{P}$ for $i=1,2$,  then $x_1 y_1 \in E$ if and only if $x_2 y_2 \in E$.  Let $X_1,X_2,Y_1,Y_2$ be the parts of~$\mathcal{P}$ containing $x_1,x_2,y_1,y_2$ respectively.  Note that since $f(x_1) = f(x_2)$ we have that $x_1,x_2$ are clones relative to $V \setminus (X_1 \cup X_2)$, and similarly $y_1,y_2$ are clones relative to $V \setminus (Y_1 \cup Y_2)$.  First suppose that $X_1 \neq Y_2$.  In this case we have $x_2 y_2 \in E$ if and only if $x_1 y_2 \in E$ (since $x_1,x_2$ are clones relative to $V \setminus (X_1 \cup X_2)$) if and only if $x_1 y_1 \in E$.  So we may assume $X_1 = Y_2$ and by similar reasoning $X_2 = Y_1$.  Now by our assumption on $f$ we may choose a vertex $x_3 \in V \setminus (X_1 \cup X_2)$ so that $f(x_3) = f(x_1) = f(x_2)$.  Now we have $x_1 y_1 \in E$ if and only if $x_3 y_1 \in E$ if and only if $x_3 y_2 \in E$ if and only if $x_2 y_2 \in E$, and this completes the proof of our claim.

Now we may define the graph $H$ with vertex set $\{1,\ldots,2m\}$ by the rule the $ij \in E(H)$ if and only if there exist vertices $x,y$ in distinct parts of~$\mathcal{P}$ for which $f(x) = i$ and $f(y) = j$ and $xy \in E$.  It follows from the above that this graph $H$ has the desired property.
\end{proof}

Our next lemma requires another concept which we introduce now.  A $k$-\emph{universal} graph is a simple graph which contains every
simple graph on $k$ vertices (up to isomorphism) as an induced
subgraph.
Alon~\cite{Alon2017} showed that there exists a small $k$-universal graph.
\begin{THM}[Alon~\cite{Alon2017}]
  For every positive integer $k$, 
  there exists a $k$-uni\-ver\-sal graph on $(1+o(1))2^{(k-1)/2}$ vertices.
\end{THM}

A $k$-\emph{loopy-universal} graph is a graph with loops but no
parallel edges. Such a graph contains every $k$-vertex graph without parallel edges as an induced subgraph up to isomorphism.  If $G =
(V,E)$ is a $k$-universal graph, then we may obtain a
$k$-loopy-universal graph $G^+$ with vertex set $V \times \{0,1\}$ by
defining $(u,i)$ to be adjacent to $(v,j)$ for $u \neq v$ whenever $uv
\in E(G)$ and adding a loop at each vertex of the form $(v,1)$.
For loopy-universal graphs this yields the following useful corollary.

\begin{COR}
  For every positive integer $k$, 
  there exists a $k$-loopy-universal graph with at most
  $(1+o(1))2^{(k+1)/2}$ vertices.
\end{COR}

We are now ready to prove that bounded rank-depth implies
bounded shrub-depth.

\begin{LEM}\label{lem:dec-to-shrub}
  For each $k$ and $r$, there exists
  \[a=(1+o(1)) 2^{(
    2^{2k+1}( 2^{2k+2}-1)
    +1)r/2}\]
such
  that 
if the cut-rank function of a simple graph $G$ has a $(k,r)$-\decomposition,
then
$G$ has an $(a,r)$-shrubbery.
\end{LEM}
\begin{proof}
Let $G = (V,E)$. Let $(T,\sigma)$ be a $(k,r)$-decomposition of
the cut-rank function $\rho_G$ of $G$.
We choose an internal node $u \in V(T)$ which has distance at most $r$
to every node in $T$.
Now modify the tree $T$ by subdividing every leaf edge so that every leaf node has distance exactly $r$ to the node $u$.  Note that the resulting tree $T$ is still a $(k,r)$-decomposition of $\rho_G$.  

Let $\ell=2^{2k+1}(2^{2k+2}-1)$.
We apply the previous corollary to choose an $\ell$-loopy-universal
graph $H$ with vertex set $C$ with $\abs{C}\le (1+o(1))2^{(\ell+1)/2}$.  We will use this graph $H$ to define for each internal vertex $v$ of the tree $T$ a function $f_v : V \rightarrow C$.  First, let $\mathcal{P}_v$ be the partition of $V$ associated with the vertex $v$ and apply Lemma~\ref{lem:nodeuniversal} to choose a graph $H_v$.  Since the graph~$H$ is $\ell$-loopy-universal, the graph $H_v$ appears as an induced subgraph of~$H$.  Therefore, we may choose a function $f_v : V \rightarrow V(H)$ with the property that whenever $x,y$ are in distinct parts of $\mathcal{P}_v$ we have $xy \in E$ if and only if $f_v(x) f_v(y) \in E(H)$.  

Now we are ready to define our function $f : V \rightarrow C^r$.  For
every vertex $x \in V$ let $u_0, u_1, \ldots, u_r$ be the vertex
sequence of the unique path in $T$ from $u$ to $x$ (so $u_0 = u$ and
$u_r = x$).  We define \[f(x) = (f_{u_{r-1}}(x), f_{u_{r-2}}(x),
\ldots, f_{u_{1}}(x), f_{u_0}(x) ).\]  Let $x,y \in V$ and assume that
they have distance $2j>0$ in $T$ (note that any two leaf vertices are
at even distance apart).  Let $v \in V(T)$ be the unique vertex of $T$
which is distance $j$ from both $x$ and $y$.  Note that $x,y$ are in
distinct parts of $\mathcal{P}_v$ and furthermore, the function $f_v$
was used to give the $j$-th coordinate of both $f(x)$ and $f(y)$.
It follows from this that $x$ and $y$ are adjacent vertices of $G$ if
and only if the $j$-th coordinate of the vectors $f(x)$ and $f(y)$
are adjacent vertices of $H$.  So, adjacency in the graph $G$ is
entirely determined by the distance function in $T$ and the function
$f : V \rightarrow C^r$ and this yields a $( (1+o(1)) 2^{(\ell+1)r/2},r)$-shrubbery for $G$, as desired.
\end{proof}

Now we are ready to prove Theorem~\ref{thm:main-sd}.
\begin{THM}\label{thm:main-sd}
A class of simple graphs has bounded rank-depth if and only if it has bounded shrub-depth.
\end{THM}
\begin{proof}%
This is an immediate consequence of Lemmas \ref{lem:shrubbery2decomp} and \ref{lem:dec-to-shrub}.
\end{proof}

Ganian, Hlin{\v{e}}n{\'y}, Ne{\v{s}}et{\v{r}}il, Obdr\v{z}\'{a}lek, and Ossona de Mendez~\cite{GHNOO2017} proved that a class of graphs of bounded shrub-depth is well-quasi-ordered by the induced subgraph relation.
This implies that for every $k$, there is a finite list of graphs such that a graph $G$ has rank-depth at most $k$ if and only if no graph in the list is isomorphic to an induced subgraph of $G$, because each minimal graph having rank-depth more than $k$ has rank-depth at most $k+1$. 
For each fixed $k$, since graphs of rank-depth at most $k$ have rank-width at most $k$, one can decide in cubic time whether an input graph has rank-depth at most $k$. For instance, it can be done by first finding a rank-decomposition of width at most $k$ by using an algorithm of Hlin\v{e}ny and Oum~\cite{HO2006} or Jeong, Kim, and Oum~\cite{JKO2018,JKO2019} and then using the theorem of Courcelle, Makowsky, and Rotics~\cite{CMR2000} to find such forbidden induced subgraphs, encoded in a monadic second-order formula. However, this approach does not provide a decomposition of width at most $k$ and we do not know such an algorithm.

Hlin\v{e}n\'y, Kwon,  Obdr\v{z}\'{a}lek, and Ordyniak~\cite{HKOO2016}
proposed a conjecture that every simple graph of sufficiently large rank-depth contains the path graph on $t$ vertices as a vertex-minor.
Very recently, Kwon, McCarty, Oum, and Wollan~\cite{KMOW2019} proved this conjecture after this paper was submitted.

\section{Depth parameters for matroids}\label{sec:matroid}
\subsection{A quick introduction to matroids}
A matroid $M=(E,\I)$ is a pair of a finite set $E$ and a set $\I$ of subsets of $E$ satisfying the following axioms:
\begin{enumerate}[({I}1)]
\item $\emptyset\in \I$,
\item If $X\in \I$ and $Y\subseteq X$, then $Y\in \I$.
\item If $X, Y\in \I$ and $\abs{X}<\abs{Y}$, then there exists $e\in Y\setminus X$ such that $X\cup \{e\}\in \I$. 
\end{enumerate}
Members of $\I$ are called \emph{independent}.
A subset $X$ of $E$ is \emph{dependent} if it is not independent.
A \emph{base} is a maximal independent set.
A \emph{circuit} is a minimal dependent set.
 Because of (I3), for every subset $X$ of $E$, all maximal independent subsets of $X$ have the same size, which is defined to be the \emph{rank}
of the set $X$, denoted by $r_M(X)$. We will omit the subscript if it
is clear from the context.
For a matroid $M=(E,\I)$, we write $E(M)$ to denote its \emph{ground set} $E$.

The rank function satisfies the \emph{submodular} inequality, namely $r(X)+r(Y)\ge r(X\cap Y)+r(X\cup Y)$ for all $X,Y\subseteq E$.
In fact, if a function $r:2^E\to \mathbb{Z}$
satisfies 
$r(\emptyset)=0$, $r(X)\le r(Y)$ for all $X\subseteq Y$
and the submodular inequality, 
then it defines a matroid $M=(E,\{X\subseteq E: r(X)=\abs{X}\})$.

The dual matroid $M^*$ of $M$ is defined as a matroid on $E$
whose set of bases is $\{ E\setminus B: \text{$B$ is a base of $M$}\}$.
A \emph{cocircuit} of $M$ is a circuit of $M^*$.
For a subset $T$ of $E$,
we define $M\setminus T$  be a matroid $(E\setminus T, \I')$
where $\I'=\{X\subseteq E\setminus T: X\in\I\}$.
This operation is called the \emph{deletion} of $T$ in $M$.
The \emph{contraction} of $T$ in $M$
is an operation to generate a matroid
$M/T=(M^*\setminus T)^*$.
If $T=\{e\}$, we write $M\setminus e$ for $M\setminus \{e\}$
and $M/e$ for $M/\{e\}$ for simplicity.
The \emph{restriction} of $M$ over $T$, denoted by $M|T$
is defined as $M|T=M\setminus (E\setminus T)$.
For more information on matroids, we refer readers to the book of Oxley~\cite{Oxley2011a}.

The \emph{connectivity function} $\lambda_M(X)$ is defined as
\[\lambda_M(X)=r_M(X)+r_M(E\setminus X)-r_M(E).\]
Then the connectivity function
of a matroid is symmetric ($\lambda_M(X)=\lambda_M(E\setminus X)$) and
submodular.
A matroid $M=(E,\I)$ is \emph{connected} if
$\lambda_M(X)>0$ for all non-empty proper subsets $X$ of $E$.
It is well known that $M$ and $M^*$ have the identical connectivity function.

The \emph{cycle matroid} $M(G)$ of a graph $G$
is the matroid on $E(G)$ such that a set of edges is independent if
it contains no cycles of $G$.
The \emph{bond matroid} $M^*(G)$ of a graph $G$ is the dual matroid of $M(G)$.

\subsection{Branch-depth, contraction depth, deletion depth, and
  con\-traction-deletion depth}
We define the \emph{branch-depth} of a matroid $M$ as the branch-depth of $\lambda_M$.
There are other ways to define depth parameters of matroids. In this section we will describe them and discuss their relations.

We define the \emph{contraction depth} $\cd(M)$, \emph{deletion depth}
$\dd(M)$, and \emph{contraction-deletion depth} $\cdd(M)$ of a matroid
$M$ as follows.

\begin{itemize}
\item If $E(M)=\emptyset$, then its contraction depth, deletion
  depth, and contraction-deletion depth are all defined to be $0$.
\item If $M$ is disconnected, then its
  contraction depth, deletion depth, and contraction-deletion depth is the
  maximum respective depth of its components.
\item If $M$ is connected and $E(M)\neq\emptyset$, then 
\begin{itemize}
  \item the contraction depth of $M$ is the minimum $k$ such that $M/e$ has contraction depth at most $k-1$ for some $e\in E(M)$,
  \item the deletion depth of $M$ is the minimum $k$ such that $M\setminus e$ has deletion depth at most $k-1$ for some $e\in E(M)$, and
  \item the contraction-deletion depth of $M$ is the minimum $k$ such that $M\setminus e$ or $M/e$ has contraction-deletion depth at most $k-1$ for some $e\in E(M)$.
\end{itemize}
\end{itemize}

Ding, Oporowski, and Oxley~\cite{DOO1995} investigated the contraction-deletion depth under the name \emph{type}.
Robertson and Seymour~\cite{RS1985} discussed similar concepts for graphs
under the names \emph{C-type} and \emph{D-type} for
contraction depth and deletion depth respectively.

As an easy corollary of Lemma~\ref{lem:conn}, we deduce the following lemma.
\begin{LEM}\label{lem:branchdepthunion}
  Let $M$ be a matroid.
  Let $k$ be the maximum branch-depth of the components of $M$.
  Then the branch-depth of $M$ is $k$ or $k+1$.
  In particular, if $M$ has at most one component having branch-depth
  exactly $k$,
  then the branch-depth of $M$ is equal to $k$.
\end{LEM}

From the definition, the following inequalities can be obtained. 
\begin{THM}\label{thm:various-depth}
  For all matroids $M$, the following hold.
  \begin{enumerate}[(1)]
  \item
    The branch-depth of $M$ is less than or equal to $\cdd(M)$.
  \item $    \cdd(M)\le \min(\cd(M),\dd(M))$.
  \end{enumerate}
\end{THM}
\begin{proof}
  Trivially $\cdd(M)\le \cd(M)$ and $\cdd(M)\le \dd(M)$. 
  For (1), we may assume that $\abs{E(M)}\ge 2$.

  We claim that if $M$ is connected, $\abs{E(M)}\ge 2$,  and $M$ has contraction-deletion depth at most $k$,
  then the branch-depth of $M$ is at most $k-1$.
  This claim, if true, implies (1) by Lemma~\ref{lem:conn}.

  To prove the claim, we proceed by induction on $\abs{E(M)}$.
  Note that since $M$ is connected and has at least two elements,
  $k\ge \cdd(M)\ge 2$.
  The statement holds trivially if $\abs{E(M)}=2$. So we may assume that $\abs{E(M)}\ge 3$.
  There exists $e\in E(M)$ such that $\cdd(M\setminus e)\le k-1$
  or $\cdd(M/e)\le k-1$.  By duality, we may assume that $\cdd(M\setminus e)\le k-1$.

  We now want to show that $M\setminus e$ has a  $(k-2,k-1)$-decomposition.
  Each component of $M\setminus e$ has contraction-deletion depth at most $k-1$
  and therefore if a component $C$ of $M\setminus e$ has at least two elements, then by the induction hypothesis, $C$ has branch-depth at most $k-2$.
  If a component  $C$ of $M\setminus e$ has only $1$ element, that is a loop or a coloop,
  then $C$ has branch-depth $0$ by definition.
  Thus all components of $M\setminus e$ have branch-depth at most $k-2$
  and by Lemma~\ref{lem:conn}, $M\setminus e$ has a $(k-2,k-1)$-decomposition.

  Let $(T,\sigma)$ be a $(k-2,k-1)$-decomposition $(T,\sigma)$ of $M\setminus e$.
  Let $v$ be an internal node of $T$ such that
  each node of $T$ is within distance $k-1$ from $v$. 
  
  Let  $T'$ be a tree obtained from $T$ by attaching a leaf $w$ to $v$
  and let $\sigma'$ be a bijection from $E=E(M)$ to the set of leaves of
  $T$,
  that extends $\sigma$
  so that $\sigma'(w)=e$.
  Then $(T',\sigma')$ is a $(k-1,k-1)$-decomposition of $M$.
  This completes the proof of the claim.
\end{proof}

Theorem~\ref{thm:various-depth} implies the following hierarchy
for a class $\mathcal M$ of matroids, depicted in Figure~\ref{fig:matroid}.
\begin{itemize}
\item 
  If $\M$ has bounded contraction-deletion depth,
  then it has bounded branch-depth.
\item If $\M$ has bounded contraction depth, then
  it has bounded contraction-deletion depth.
\item If $\M$ has bounded deletion depth, then
  it has bounded contraction-deletion depth.
\end{itemize}

Here are examples disproving the converse of each of these.
We omit the proof.

\begin{itemize}
\item
  An \emph{$(m,n)$-multicycle} $C_{m,n}$ is the graph obtained
  from the cycle graph $C_m$ by replacing each edge with $n$ parallel edges.
  Then $M(C_{n,n})$ has branch-depth  $2$
  and yet its contraction-deletion depth is $n$, shown by
  Dittmann and Oporowski~\cite[Lemma 2.4]{DO2002}.
  By duality, $M^*(C_{n,n})$ has deletion-depth $n$ and branch-depth $2$.
  
\item The matroid $U_{n-1,n}=M(C_n)$ has deletion depth $2$, contraction-deletion depth $2$
  and contraction depth $n+1$.
  Its dual $U_{1,n}=M^*(C_n)$ has contraction depth $2$, contraction-deletion depth $2$
  and deletion depth $n+1$.
\end{itemize}

\subsection{Relation to matroid minors}

The branch-depth of matroids cannot increase by taking a minor.
A matroid $N$ is a \emph{minor} of a matroid $M$ if $N=M\setminus X/Y$ for some disjoint subsets $X$, $Y$ of $E(M)$.
\begin{PROP}\label{prop:matroidminor}
  If $N$ is a minor of a matroid $M$, then
  the branch-depth of $N$ is less than or equal to the branch-depth of $M$.
\end{PROP}
\begin{proof}
  We may assume that $N$ has at least two elements and $\abs{E(M)}>\abs{E(N)}$.
  Let $k$ be the branch-depth of $M$.
  Let $(T,\sigma)$ be a $(k,k)$-\decomposition{} of $M$.
  Let $T'$ be a minimal subtree of $T$ containing all leaves corresponding to $E(N)$ by $\sigma$.
  Let $\sigma'$ be the restriction of $\sigma$ on the set of leaves of $T'$.
  As $\abs{E(M)}>\abs{E(N)}\ge 2$,
  $T'$ must have at least one internal node.
  Clearly the radius of $T'$ is at most $k$.

  It is well known that 
  for all $X\subseteq E(M)$,
  $\lambda_N(X\cap E(N))\le \lambda_M(X)$ (see Oxley~\cite[Corollary 8.2.5]{Oxley2011a}).
  This implies that the width of $(T',\sigma')$ is at most $k$.
  Thus, $(T',\sigma')$ is a $(k,k)$-decomposition of $N$.
\end{proof}
Proposition~\ref{prop:matroidminor} allows us to deduce the following algorithm. 
\begin{COR}
  For each fixed finite field $\mathbb F$ and an integer $k$, we can decide in time $O(n^3)$ whether the input $n$-element rank-$r$ matroid represented by an $r\times n$ matrix over $\mathbb F$ has branch-depth at most $k$. 
\end{COR}
\begin{proof}
  Trivially if the branch-depth of a matroid is at most $k$, then 
  the branch-width of a matroid is at most $k$.
  Also, Geelen, Gerards, and Whittle~\cite{GGW2002} proved that matroids over $\mathbb F$ of bounded branch-width are well-quasi-ordered under the minor relation and therefore there is a finite list of matroids such that a matroid $M$ representable over $\mathbb F$ has branch-depth at most~$k$ if and only if no matroid in the list is isomorphic to a minor of~$M$.
  So we can use the algorithm of Hlin\v{e}n\'y and Oum~\cite{HO2006} or Jeong, Kim, and Oum~\cite{JKO2018,JKO2019} to find a branch-decomposition of width at most~$k$ if it exists in time $O(n^3)$. 
  If there is no such branch-decomposition, then the branch-depth is larger than $k$.
  If we have a branch-decomposition of width at most~$k$, then we use the algorithm of Hlin\v{e}n\'y~\cite{Hlineny2004} to decide whether the input matroid represented by a matrix over $\mathbb F$ has a minor isomorphic to a fixed matroid. 
\end{proof}

However, we will show that deletion depth, contraction depth,
and con\-trac\-tion-deletion depth may increase
by taking a minor.
First we will present an example showing
that the deletion depth of a minor of $M$
is not necessarily less than or equal to the deletion depth of $M$.
By duality, this implies that 
the contraction depth of a minor of $M$
is not necessarily less than or equal to the contraction depth of $M$.

Let $K_3^+$ be the graph in Figure~\ref{fig:dd}
and let $e$ be the edge of $K_3^+$ shown in Figure~\ref{fig:dd}.
Then $M(K_3^+\setminus e)$ has two components, each having two elements.
Thus, $\dd(M(K_3^+))\le 3$.
It is easy to check that $\dd(M(K_3^+))\ge 3$
and therefore $\dd(M(K_3^+))=3$.
However, $\dd(M(K_3^+)/e)=4$ because there is no way to break it into more than one component by deleting at most three edges.

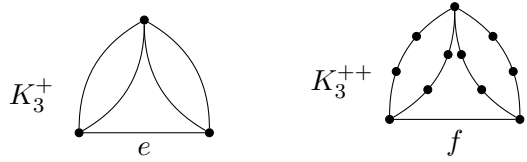
\begin{figure}
  \centering
 \tikzstyle{v}=[circle, draw, solid, fill=black, inner sep=0pt, minimum width=3pt]
  
 \begin{tikzpicture}[scale=1]
   \node at (-1.5,0) {$K_3^+$};
   \foreach \x in {0,1,2}{
     \node[v] (v\x) at (120*\x+90:1) {};
   }
   \foreach \y in {1,2}{
     \draw (v0) to [bend right] (v\y);
     \draw (v0) to [bend left] (v\y);
   }
   \draw (v1) --(v2)  node [pos=0.5,below]{$e$};
 \end{tikzpicture}\quad\quad\quad
 \begin{tikzpicture}[scale=1]
   \node at (-1.5,0) {$K_3^{++}$};
   \foreach \x in {0,1,2}{
     \node[v] (v\x) at (120*\x+90:1) {};
   }
   \foreach \y in {1,2}{
     \draw (v0) to [bend right]    node [pos=0.3333,v] {} node[pos=0.6667,v]{} (v\y);
     \draw (v0) to [bend left]    node [pos=0.3333,v] {} node[pos=0.6667,v]{} (v\y);
   }
   \draw (v1) --(v2)  node [pos=0.5,below]{$f$};
 \end{tikzpicture}
  \caption{Graphs $K_3^+$ and $K_3^{++}$.}
  \label{fig:dd}
\end{figure}

Now let us present an example for the contraction-deletion depth.
Let $K_3^{++}$ be the graph in Figure~\ref{fig:dd} and let $f$ be the edge of $K_3^{++}$.
Then $M(K_3^{++}\setminus f)$ has two components, each having the contraction-deletion depth~$2$
and therefore the contraction-deletion depth of $M(K_3^{++})$ is~$3$.
However, in $M(K_3^{++}/f)$,
no matter how we delete or contract any two elements,
there will be a circuit of size $6$ and therefore the contraction-deletion depth of $M(K_{3}^{++}/f)$ is at least $4$. Indeed, it is easy to check that the contraction-deletion depth of $M(K_3^{++}/f)$ is exactly $4$.

\subsection{Large circuits and contraction depth}
Let us first discuss obstructions for small contraction depth.
Here is a theorem due to Seymour, see~\cite{DOO1995}.
\begin{THM}[Seymour~(see \cite{DOO1995})]\label{thm:seymour}
  If $C$ is a longest circuit in a connected matroid $M$, then $M/C$ has no circuits of size at least $\abs{C}$.
\end{THM}

\begin{LEM}\label{lem:circuithalf}
  Let $M$ be a matroid and $e\in E(M)$.
  If $C$ is a circuit of a matroid $M$ with $\abs{C}\ge 2$, then $M/e$ has a circuit of size at least $\abs{C}/2$.
\end{LEM}
\begin{proof}
  We may assume that $E(M)=C\cup \{e\}$. If $e\in C$, then $C\setminus \{e\}$
  is a circuit of $M/e$. Thus we may assume $e\notin C$. If $D_1$,
  $D_2$ are circuits of $M/e$ and $D_1, D_2\neq C$, then
  $D_1\cup\{e\}$, $D_2\cup \{e\}$ are circuits of $M$. By the circuit
  elimination axiom, $M$ should have a circuit $F\subseteq D_1\cup
  D_2$ and therefore $D_1\cup D_2=C$. Thus if $M/e$ has at least two
  distinct circuits, then one of them has size at least $\abs{C}/2$.
  Thus we may assume that $M/e$ has a unique circuit $D$. If $D\neq C$, then $D\cup\{e\}$ and $C$ are circuits of $M$. Let $f\in D$. Then $M$ has a circuit $D'\subseteq (D\cup C\cup\{e\})\setminus \{f\}$. This means that $e\in D'$ and $D'\setminus \{e\}$ is a circuit of $M/e$. This contradicts our assumption that $M/e$ has a unique circuit.
\end{proof}

Now we show that having small contraction depth is equivalent to having no large circuit.
In the following theorem, the upper bound is based on Ding, Oporowski, and Oxley~\cite{DOO1995}, though they state it in terms of the contraction-deletion depth. We include its proof for the completeness.
\begin{THM}\label{thm:cd}
  Let $c$ be the length of a largest circuit  in $M$. 
  (If $M$ has no circuits, then let $c=1$.)
  Then 
  \[
  \log_2 c\le \cd(M)\le c(c+1)/2.
  \]
\end{THM}
\begin{proof}
  For the upper bound, we proceed by induction on $c$. 
  Observe that if $c\le 1$, then its contraction depth is at most $1$, as each component has at most $1$ element.
  We may assume that $M$ is connected and $c>1$.  Let $C$ be a longest circuit of $M$. 
  Then $\cd(M)\le \abs{C}+\cd(M/C)$. By Theorem~\ref{thm:seymour} and the induction hypothesis, $\cd(M/C)\le c(c-1)/2$ and therefore $\cd(M)\le c(c+1)/2$.

  Now let us prove the lower bound. We apply the induction on $\abs{E(M)}$.
  We may assume that $M$ is connected and $c>1$.
  So we may assume that $\abs{E(M)}>1$.
  By Lemma~\ref{lem:circuithalf}, for all $e\in E(M)$, $M/e$ has a circuit of size at least $c/2$ and therefore $\cd(M/e)\ge \log_2 c -1$.
  We conclude that $\cd(M)\ge \log_2 c$.
\end{proof}
By duality, we easily obtain the following.
\begin{COR}\label{cor:dd}
  Let $c^*$ be the length of a largest cocircuit in $M$. 
  (If $M$ has no cocircuits, then let $c^*=1$.)
  Then 
  \[
  \log_2 c^*\le \dd(M)\le c^*(c^*+1)/2.
  \]
\end{COR}
Theorem~\ref{thm:cd} and Corollary~\ref{cor:dd}
prove the following theorem.
\begin{THM}\label{thm:main-long-circuit}
  \begin{enumerate}[(i)]
  \item   A class of matroids has bounded contraction depth if and only if
  all circuits have bounded size.
  \item   A class of matroids has bounded deletion depth if and only if
  all cocircuits have bounded size.
  \end{enumerate}
\end{THM}

By Theorem~\ref{thm:various-depth}, we deduce the following.
\begin{COR}\label{cor:no-long-circuit}
  Let $k\ge 1$. 
  If a matroid $M$ has no circuits of size more than $k$
  or no cocircuits of size more than $k$,
  then the branch-depth of $M$ is at most $\frac12{k(k+1)}$.
\end{COR}

This also allows us to characterize classes of matroids having
bounded contraction depth and bounded deletion depth at the same time.
\begin{COR}\label{cor:cddd}
  A class $\mathcal M$ of matroids has bounded deletion depth and bounded contraction depth
  if and only if
  there exists $m$ such that 
  every connected component of a matroid in $\mathcal M$ has at most $m$ elements.
\end{COR}
\begin{proof}
  The converse is trivial.
  Let us prove the forward direction.
  Theorem~\ref{thm:cd} and Corollary~\ref{cor:dd}
  imply that there exist $c$ and $c^*$ such that
  in every connected component of a matroid in $\mathcal M$,
  all circuits have size at most $c$
  and all cocircuits have size at most $c^*$.
  Lemos and Oxley~\cite{LO2001} showed that
  if a connected matroid
  has no circuits of more than $c$ elements
  and no cocircuits of more than $c^*$ elements,
  then it has at most $cc^*/2$ elements.
  Thus, each component of a matroid in $\mathcal M$ has at most $cc^*/2$ elements.
\end{proof}

It is natural to ask an obstruction for having large branch-depth in a matroid.
The \emph{$n$-fan} $F_n$ is the graph on $n+1$ vertices having a vertex $v$ adjacent to all other vertices such that $F_n\setminus v$ is a path on $n$ vertices.
We conjecture that matroids of sufficiently large branch-depth
has $M(F_n)$ or
the uniform matroid $U_{n,2n}$ of rank $n$ on $2n$ elements as a minor.
We also conjecture that matroids of sufficiently large contraction-deletion depth
has $M(F_n)$,
the uniform matroid $U_{n,2n}$,
$M(C_{n,n})$, or its dual $M^*(C_{n,n})$ as a minor.

For a graphic matroid, Dittmann and Oporowski~\cite{DO2002} identified obstructions
for large contraction-deletion depth. Here is their theorem in terms of our terminologies.
\begin{THM}[Dittmann and Oporowski~\cite{DO2002}]
  For every integer $n>3$, there exists $N$ such that
  every graph whose cycle matroid has con\-tract\-ion-deletion depth at least $N$
  has a minor isomorphic to $F_n$, $C_{n,n}$, or $C_{n,n}^*$.
\end{THM}

\subsection{Connections to the tree-depth of a graph}
\label{subsec:graphic}

Hicks and McMurray Jr.~\cite{HM2005} and independently, Mazoit and Thomass\'e~\cite{MT2005} proved that if a graph $G$ has at least one non-loop cycle,
then the branch-width of its cycle matroid $M(G)$ is equal to the branch-width of $G$.
How about branch-depth? Theorem~\ref{thm:main-td} shows that the branch-depth and the tree-depth are tied for graphs.
Can we say that the tree-depth of graphs
are tied to the branch-depth of their cycle matroids?

It turns out that the tree-depth of graphs is not tied to
the branch-depth of their cycle matroids in general. 
This is because connectedness in graphs is different from connectedness in their cycle matroids.
For example, the path~$P_n$ on $n>2$ vertices has tree-depth $\lceil \log_2(n+1)\rceil$, see \cite[(6.2)]{NO2012}.
However its cycle matroid $M(P_n)$ has contraction depth~$1$, deletion depth~$1$, and branch-depth $1$.

Only one direction holds in general; if the tree-depth of a graph is small,
then the contraction depth, the contraction-deletion depth, and the branch-depth of
its cycle matroid are small, by the following proposition.

\begin{PROP}\label{prop:td-to-cd}
  Let $t$ be the tree-depth of a graph $G$.
  Then,
  \[
    \bd(M(G))
    \le  \cdd(M(G))\le \cd(M(G))\le 2^{2(t-1)}
  \]
  and
  \[    \bd(M(G))\le \bd(G)-1\le t.\]
  In addition, if $G$ is connected, then $\bd(M(G))\le \bd(G)-1\le t-1$.
\end{PROP}
\begin{proof}
  For the first inequality, 
  it is enough to show that $\cd(M(G))\le t$ by Theorem~\ref{thm:various-depth}.
  Let $L$ be the length of a longest cycle of $G$.
  Since $t\ge \td(C_L)=1+\td(P_{L-1})=1+\lceil \log_2L  \rceil$,
  we deduce that $L\le 2^{t-1}$.
  By Theorem~\ref{thm:cd}, we deduce that
  $\cd(M(G))\le 2^{t-1}(2^{t-1}+1)/2\le 2^{2(t-1)}$.

  For the second inequality,
  we use the following inequality in \cite[Lemma 8.1.7]{Oxley2011a}:
  \begin{quote}
    For a connected graph $G$, if $\emptyset\neq X\neq E(G)$, then
    \[\lambda_{M(G)}(X)\le \lambda_{G}(X)-1.\]
  \end{quote}
  This inequality implies that if $G$ is connected, then
  $\bd(M(G))\le \bd(G)-1\le t-1$ by Lemma~\ref{lem:td2bd}.
  If $G$ is disconnected, then we identify one vertex from each component into one vertex. That cannot increase the branch-depth of the graph
  but the cycle matroid does not change and therefore we deduce our conclusion
  with Theorem~\ref{thm:main-td}.
\end{proof}

Now we will show that under some mild connectivity assumptions,
the tree-depth of graphs and
the contraction depth of their cycle matroids are tied.
For the tree-depth, it is well known that every graph of tree-depth larger than $n$
contains a path of length $n$, see~\cite[Proposition 6.1]{NO2012}.
An old theorem of Dirac~\cite{Dirac1952} states that a $2$-connected graph with a sufficiently long path
contains a long cycle. (In fact the paper contains a proof for $\sqrt{L/2}$ and remarks that it can be improved to $2\sqrt{L}$.)
\begin{THM}[Dirac~\cite{Dirac1952}]\label{thm:dirac}
  If a $2$-connected graph has a path of length $L$,
  then it has a cycle of length at least  $2\sqrt{L}$.
\end{THM}
By combining with Proposition~\ref{prop:td-to-cd},
we can deduce that the tree-depth of $2$-connected graphs
are tied to the contraction depth of their cycle matroids.
\begin{PROP}
  Let $G$ be a $2$-connected graph of tree-depth $t$. Then
  \[
    1+\frac{1}{2}\log_2 (t-1)\le \cd(M(G))\le 2^{2(t-1)}.
  \]
\end{PROP}
\begin{proof}
  By \cite[Proposition 6.1]{NO2012}, $G$ has a path of length $t-1$
  and by Theorem~\ref{thm:dirac}, $G$ contains a cycle of length at least  $2\sqrt{t-1}$. This means that the contraction depth of $M(G)$ is at least $\log_2(2\sqrt{t-1})$
  by Theorem~\ref{thm:cd}.
  The upper bound is given by 
  Proposition~\ref{prop:td-to-cd}. 
\end{proof}

Still, $2$-connectedness does not imply that the tree-depth of graphs
and the branch-depth of their cycle matroids are tied; for instance $C_n$.
However, for $3$-connected graphs,
tree-depth of graphs
and contraction depth, contraction-deletion depth, branch-depth of their cycle matroids are all tied each other.
To prove that, we use the following theorem.

\begin{THM}[Ding, Dziobiak, and Wu~{\cite[Proposition 3.8]{DDW2016}}]\label{thm:ddw}
  Let $G$ be a $3$-connected graph with a path of length $L$.
  Then $G$ has a minor isomorphic to  the wheel graph $W_k$
  with $k=\lfloor \frac1{6\sqrt2} \sqrt{\log (2L/5)}\rfloor$.
\end{THM}

We will see that indeed, if a matroid contains the cycle matroid of a large fan
as a minor, then it has large branch-depth, thus implying that it has large contraction-deletion depth.
To see this, we will first show a connection between branch-depth of a binary matroid
and rank-depth of a simple bipartite graph.
The \emph{fundamental graph} of a binary matroid $M$ with respect to a base $B$
is a simple bipartite graph on $E(M)$ such that
$e\in B$ is adjacent to $f\in E(M)\setminus B$ if and only if
$(B\setminus\{e\})\cup \{f\}$ is independent.
It is well known that if $G$ is a fundamental graph of a binary matroid $M$,
then the cut-rank function of $G$ is identical to the connectivity function of $M$, see Oum~\cite{Oum2004}. Thus we deduce the following lemma.
\begin{LEM}\label{lem:fund}
  If $G$ is a fundamental graph of a binary matroid $M$, then
  the branch-depth of $M$ is equal to the rank-depth of $G$.
\end{LEM}
Note that $P_{2n-1}$ is a fundamental graph of the cycle matroid of $F_n$.
By Proposition~\ref{prop:path},
$\bd(M(F_n))>\log (2n-1) / \log (1+4\log (2n-1))$ for $n\ge 2$.  This implies a weaker variant of the following theorem due to 
Dittmann and Oporowski~\cite{DO2002}, showing that 
if a graph has a large fan as a minor, then the contraction-deletion depth
of its cycle matroid is large as follows.

\begin{THM}[Dittmann and Oporowski~{\cite[Theorem 1.4]{DO2002}}]\label{thm:do}
  If a graph $G$ contains $F_n$ as a minor, then
  the contraction-deletion depth of $M(G)$ is at least $\lfloor\log_2 n \rfloor+1$.
\end{THM}

Now we are ready to show that
for cycle matroids of $3$-connected graphs,
the tree-depth of graphs and 
the contraction depth, the contraction-deletion depth, and the branch-depth of their cycle matroids
are all tied each other.
\begin{PROP}
  Let $G$ be a $3$-connected graph of tree-depth $t$
  and let $k=\lfloor \frac1{6\sqrt{2}} \sqrt{\log (2(t-1)/5)}\rfloor$. Then
  \[
    \frac{\log (2k-1) }{\log (1+4\log (2k-1))}
    \le 
    \bd(M(G))
    \le 
      \cdd(M(G))\le \cd(M(G))\le 2^{2(t-1)}
    \]
    and
    \[\bd(M(G))\le t-1.\]
\end{PROP}
\begin{proof}
  This is trivial
  from Theorems~\ref{thm:ddw}
  and Propositions~\ref{prop:td-to-cd} and 
  because  $W_k$ contains $F_k$ as a minor.
\end{proof}

\subsection{Rank-depth is less than or equal to tree-depth}
Oum~\cite{Oum2006c} showed that the rank-width of a simple graph is less than or equal to the branch-width, by a reduction using matroids.
By the same method, we are going to show the following.
The \emph{incidence graph} of a graph $G$, denoted by $I(G)$, is the subdivision
of $G$ obtained by subdividing every edge of $G$ exactly once.
\begin{THM}\label{thm:rd-to-td}
  Let $G$ be a simple graph of tree-depth $t$. Then
  \[
    \rd(G)\le \rd(I(G))\le t.
  \]
\end{THM}
\begin{proof}
  Let $G$ be a simple graph of tree-depth $t$.
  Let $G'$ be a simple graph obtained from $G$ by adding a new vertex  $v$ adjacent to
  all vertices of $G$.
  Then $\td(G')\le t+1$.
  Let $T$ be a spanning tree of $G'$ consisting of all edges incident with $v$.
  Then $I(G)$ is the fundamental graph of $M(G')$  with respect to $E(T)$.
  It is easy to see that $G$ is a vertex-minor of $I(G)$.
  By Proposition~\ref{prop:td-to-cd} and Lemma~\ref{lem:fund}, 
  $\rd(I(G))=\bd(M(G'))\le \td(G')-1\le t$ because $G'$ is connected.
  By Lemma~\ref{lem:vertexminor},
  $\rd(G)\le \rd(I(G))\le t$.
\end{proof}

\section{Well-quasi-ordering of matroids of bounded contraction depth}
\label{sec:wqo}
In this section, we aim to prove that matroids representable over a fixed finite field having no large circuits are well-quasi-ordered by the matroid restriction.
\begin{THM}\label{thm:main-wqo}
  Let $\F$ be a finite field.
  Every class of $\F$-representable matroids of bounded contraction depth
  is well-quasi-ordered by restriction.
\end{THM}

The assumption on contraction depth is clearly necessary, because
$M(C_3)$, $M(C_4)$, $\ldots$ form an anti-chain.
We also need the condition on the representability over a fixed finite field;
there is an infinite anti-chain of matroids having bounded contraction depth
in Ding, Oporowski, and Oxley~\cite{DOO1995}.

To prove Theorem~\ref{thm:main-wqo}, we prove a stronger statement in terms of matrices.
For an $X\times Y$ matrix $N$ over a field $\F$, we write $E(N)=Y$ to denote the set of indexes for columns. For a subset $Y'$ of $Y$, we write $N[Y']$ to denote the submatrix of $N$ induced by taking columns indexed by $Y'$.
Two matrices $N_1$ and $N_2$ are \emph{isomorphic} if there exists a bijection  $\phi:E(N_1)\to E(N_2)$, called an \emph{isomorphism},  such that the column vector of $e$ in $N_1$ for $e\in E(N_1)$ is equal to the column vector of $\phi(e)$ in $N_2$. 
For a matrix $N$, we write $M(N)$ to denote the matroid represented by $N$,
meaning that the columns of $N$ are indexed by the elements of $E(M)$
and a set is independent in $M(N)$ if its corresponding set of column vectors is linearly independent.
The \emph{contraction depth} of  a matrix $N$ is defined to be the contraction depth of $M(N)$.

We say that a matrix $N_1$ is a \emph{restriction} of a matrix $N_2$ if we can apply elementary row operations to  $N_2[Y]$ for some $Y\subseteq E(N_2)$ and delete some zero rows to obtain  a matrix isomorphic to $N_1$.
Clearly if a matrix $N_1$ is a restriction of $N_2$, then $M(N_1)$ is a restriction of $M(N_2)$. (A matroid $N$ is a \emph{restriction} of a matroid $M$ if $N$ is obtained from $M$ by deleting some elements.)
A matrix is said to have \emph{full rank} if its row vectors are linearly independent.

Let $(Q,\le)$ be a quasi-order.
A \emph{$Q$-labeling} of a matrix $N$ is a function $f$ from $E(N)$ to $Q$. The pair $(N,f)$ is called a \emph{$Q$-labeled matrix}.
For two $Q$-labeled matrices $(N_1,f_1)$ and $(N_2,f_2)$, we write $(N_1,f_1)\preceq (N_2,f_2)$ if $N_1$ is a restriction of $N_2$ with an injective function $\phi:E(N_1)\to E(N_2)$ such that $f_1(e)\le f_2(\phi(e))$ for all $e\in E(N_1)$.

Let $\N^\F_k$ be the set of all matrices over $\F$ having full rank and contraction depth at most $k$ and
let $\N^\F_k(Q)$ be the set of all $Q$-labeled matrices over $\F$ having full rank and having contraction depth at most~$k$.

Now we are ready to state a proposition on well-quasi-ordering of $Q$-labeled matrices.

\begin{PROP}\label{prop:wqomatrix}
  Let $\F$ be a finite field. Then, $(\N^\F_k(Q),\preceq)$ is a well-quasi-order if $(Q,\le)$ is a well-quasi-order.
\end{PROP}

Before presenting the proof of Proposition~\ref{prop:wqomatrix},
we will discuss issues related to connectedness in matroids.
We say that two matrices are \emph{row equivalent} if one is obtained from the other
by a sequence of elementary row operations.
If two matrices $N_1$ and $N_2$ are row equivalent, then
$M(N_1)=M(N_2)$ and therefore $M_1\in\N^\F_k$ if and only if
$M_2\in\N^\F_k$.
Furthermore by the definition of the restriction,
a matrix $N'$ is a restriction of $N_1$
if and only if $N'$ is a restriction of $N_2$.

We say that  a matrix $N$ is the \emph{disjoint union} of two matrices $N_1$ and $N_2$
if \[N =
  \bordermatrix {
    & E(N_1) &   E(N_2)\cr
    & N_1 & 0 \cr
    & 0 & N_2
  }.\]
We say that a matrix is \emph{connected}
if it is not row equivalent to the disjoint union of two matrices.
(Permuting columns do not change matrices for us because columns are indexed by
their indices.)
It is easy to see the following.
\begin{LEM}
  Let $N$ be a full-rank matrix over $\F$.
  Then the matroid $M(N)$ is disconnected
  if and only if
  there exist two full-rank matrices $N_1$ and $N_2$
  such that $N$ is row equivalent to the disjoint union of $N_1$ and $N_2$. 
\end{LEM}
Like matroids, if a full-rank matrix $N$ is the disjoint union
of $N_1$, $N_2$, $\ldots$, $N_k$, then
we may call each $N_i$ a \emph{component} of $N$.
The following trivial lemma together with  Higman's lemma~\cite{Higman1952}
will allow us to reduce the proof of Proposition~\ref{prop:wqomatrix}
to connected matrices.
\begin{LEM}\label{lem:disjointunionmat}
  Let $N_1$, $N_2$, $N_1'$, $N_2'$ be
  full-rank matrices over $\F$.
  Let $N$ be a matrix row equivalent to the disjoint union of $N_1$ and $N_2$
  and $N'$ be a matrix row equivalent to the disjoint union of  $N_1'$ and $N_2'$.
  If $N_1'$ is a restriction of $N_1$
  and $N_2'$ is a restriction of $N_2$,
  then $N'$ is a restriction of $N$.
\end{LEM}
\begin{proof}
  We may assume that $N$ is the disjoint union of $N_1$ and $N_2$.
  We deduce the conclusion because the disjoint union of $N_1'$ and $N_2'$
  is clearly  a restriction of $N$.
\end{proof}
\begin{proof}[Proof of Proposition~\ref{prop:wqomatrix}]
  We proceed by induction on $k$. If $k=0$, then
  the only matrix of contraction depth at most $k$ is a zero matrix.
  But there is only one zero matrix of full rank.
  And so $\N^\F_0$  is trivially well-quasi-ordered by $\preceq$.
  So we may assume $k>0$.

  By Higman's lemma \cite{Higman1952} and Lemma~\ref{lem:disjointunionmat},
  it is enough to prove that connected matrices in $\N^\F_k(Q)$ are well-quasi-ordered
  by $\preceq$.
  Let $(N_1,f_1)$, $(N_2,f_2)$, $\ldots$ be an infinite sequence of $Q$-labeled connected matrices in $\N^\F_k(Q)$. We claim that there exist $i<j$ such that $(N_i,f_i)\preceq (N_j,f_j)$. By the induction hypothesis, we may assume that all $N_i$ have contraction depth exactly $k$. 

  Let $e_i$ be an element of $E(N_i)$ such that each component of $M(N_i)/e_i$ has contraction depth at most $k-1$. By taking an infinite subsequence, we may assume that $f_1(e_1)\le f_2(e_2)\le f_3(e_3)\le \cdots$.

  Now let us construct a matrix representing $M(N_i)/e_i$.
  Because $M(N_i)$ is connected, $\{e_i\}$ is independent in $M(N_i)$. By applying elementary row operations, we may assume that the column vector for $e_i$ is represented by a vector with one $1$ on the top row and $0$'s on all other rows. 
  We may further assume by applying elementary row operations to non-top rows such that $N_i$ is of the following form
\[
N_i=
\begin{pmatrix}
  1  &  ? \\
  0 & N_{i}'
\end{pmatrix}.
\]
As $N_i$ has full rank, $N_i'$ has full rank.
By the choice of $e_i$, each $N_{i}'$ has contraction depth at most $k-1$ and therefore $N_{i}'\in \N^\F_{k-1}$.

Let $Q'=Q\times \F$ and let us define an order $\le$ on $Q'$ such that $(x_1,x_2)\le (y_1,y_2)$ if and only if $x_1\le y_1$ and $x_2=y_2$. Then $(Q', \le)$ is a quasi-order if $(Q,\le)$ is a quasi-order.

For $e\in E(N_i)\setminus\{e_i\}$, 
let us define $f_{i}'(e)=(f_{i} (e), (N_i)_{1,e})\in Q'=Q\times \F$ where $(N_i)_{1,e}$ means the entry of $N_i$ in the top row and the column of $e$.
Then $(N_{i}', f_{i}')\in \N^\F_{k-1}(Q')$.
  
By the induction hypothesis, the set $\{(N_i,f_{i}'):i\in \{1,2,\ldots\}\}\subseteq \N_{k-1}^\F(Q')$ is well-quasi-ordered by $\preceq$.
Thus, there exist $i<j$ such that $(N_{i}',f_{i}')\preceq (N_{j}',f_{j}')$.
It follows easily that $(N_i,f_i)\preceq (N_j,f_j)$.
\end{proof}

Now we will see why Proposition~\ref{prop:wqomatrix} implies Theorem~\ref{thm:main-wqo} in a stronger form.
A $Q$-labeled matroid is a pair $(M,f)$ of a matroid $M$ and $f:E(M)\to Q$. A $Q$-labeled matroid $(M_1,f_1)$ is a \emph{restriction} of a $Q$-labeled matroid $(M_2,f_2)$ if $M_1$ is isomorphic to $M_2|X$ for some $X$ with an isomorphism $\phi$ such that $f_1(e)\le f_2(\phi(e))$ for all $e\in M$. 

Now we are ready to prove Theorem~\ref{thm:main-wqo}.
\begin{proof}
  We prove a stronger statement. Let $(Q,\le)$ be a well-quasi-order. We claim that
  $Q$-labeled $\F$-representable matroids of contraction depth at most $k$ are well-quasi-ordered by restriction. 

  Suppose that $(M_1,f_1)$, $(M_2,f_2)$, $\ldots$ are $Q$-labeled $\F$-representable matroids of contraction depth at most $k$.
  For each $i$, we can choose a matrix~$N_i$ over $\F$ having full rank such that $M(N_i)=M_i$. By Proposition~\ref{prop:wqomatrix}, there exist $i<j$ such that $(N_i,f_i)\preceq (N_j,f_j)$. It follows that $(M_i,f_i)$
  is isomorphic to a restriction of $(M_j,f_j)$.
\end{proof}

Theorem~\ref{thm:main-wqo} allows us to characterize
well-quasi-ordered classes of $\F$-representable matroids over any
fixed finite field $\F$,
because 
$\{U_{n-1,n}:n=2,3,\ldots\}=\{M(C_n):n=2,3,\ldots\}$ is an anti-chain with respect to the restriction.
\begin{COR}
  Let $\F$ be a finite field.
  Let $\mathcal I$ be a class of $\F$-representable matroids closed under taking the restriction.
  Then the following are equivalent.
  \begin{enumerate}[(i)]
  \item   $\mathcal I$ is well-quasi-ordered by restriction.
  \item 
    $\{U_{n-1,n}:n=2,3,\ldots\}\not\subseteq \mathcal I$.
  \item The contraction depth of $\mathcal I$ is bounded.
  \end{enumerate}
\end{COR}

\providecommand{\bysame}{\leavevmode\hbox to3em{\hrulefill}\thinspace}
\providecommand{\MR}{\relax\ifhmode\unskip\space\fi MR }
\providecommand{\MRhref}[2]{%
  \href{http://www.ams.org/mathscinet-getitem?mr=#1}{#2}
}
\providecommand{\href}[2]{#2}

\end{document}